\documentclass[12pt,a4paper,oneside,reqno]{amsart}

\usepackage{amsmath}

\usepackage{pstricks}
\usepackage{pst-plot}

\IfFileExists{amsfonts.sty}{
   \usepackage{amsfonts}
   \newcommand{\inBlackBoardFont}{\mathbb}}{
   \newcommand{\inBlackBoardFont}{\mathbf}}

\IfFileExists{mathrsfs.sty}{
   \usepackage{mathrsfs}
   \newcommand{\inScriptFont}{\mathscr}}{
   \newcommand{\inScriptFont}{\mathfrak}}

\newcommand{\C}{{\inBlackBoardFont C}}
\newcommand{\R}{{\inBlackBoardFont R}}

\newcommand{\Acl}{{\inScriptFont A}}

\newcommand{\acl}{{\mathfrak a}}

\newcommand*{\Hcc}[1]{H^*_{#1}}
\newcommand*{\Hcb}[1]{(H^*_{#1})'}

\IfFileExists{mathtools.sty}{
   \usepackage{mathtools}}{
   \providecommand{\coloneqq}{\mathrel{\mathop:}=}}

\newcommand{\dd}{\,d}
\newcommand{\ds}{\dd s}
\newcommand{\dx}{\dd x}

\DeclareMathOperator{\Real}{Re}

\DeclareMathOperator{\Lip}{Lip}

\newcommand*{\Hypo}[1]{\textup{($\mathbf{#1}$)}}

\newcommand{\parameter}{\,\cdot\,}

\newcommand*{\Norm}[1]{\mathopen\lVert#1\mathclose\rVert}
\newcommand*{\Abs}[1]{\mathopen\lvert#1\mathclose\rvert}
\newcommand*{\Curly}[1]{\mathopen\{#1\mathclose\}}
\newcommand*{\Scalr}[1]{(#1)}
\newcommand*{\Scalb}[1]{[#1]}

\IfFileExists{cite.sty}{\usepackage{cite}}{}

\providecommand{\bysame}{\leavevmode\hbox to3em{\hrulefill}\thinspace}

\numberwithin{equation}{section}
\numberwithin{figure}{section}

\usepackage{amsthm}
\newtheorem{theorem}{Theorem}[section]
\newtheorem{lemma}[theorem]{Lemma}
\newtheorem{proposition}[theorem]{Proposition}
\newtheorem{corollary}[theorem]{Corollary}
\theoremstyle{definition}
\newtheorem{definition}[theorem]{Definition}
\newtheorem{remark}[theorem]{Remark}
\newtheorem{example}[theorem]{Example}

\let\originalref\ref
\renewcommand*{\ref}[1]{\textup{\originalref{#1}}}

\newcommand*{\iref}[1]{\textup{(\ref{#1})}}
\newcommand*{\iiref}[1]{\textup{(\ref{#1})}}

\begin{document}
\title[Stability for Semilinear Parabolic Problems]{Stability
for Semilinear Parabolic Problems in $L_2$, $W^{1,2}$, and
interpolation Spaces}

\author{Pavel Gurevich}
\address{Pavel Gurevich\\
Free University of Berlin\\
Dept.\ of Mathematics (WE1)\\
Arnimallee 3\\
D-14195 Berlin\\
Germany;
Peoples' Friendship University of Russia\\
Miklukho-Maklaya 6\\
117198 Moscow\\
Russia}
\email{gurevichp@gmail.com}

\author{Martin V\"{a}th}
\address{Martin V\"{a}th\\
Mathematical Institute\\
Academy of Sciences of the Czech Republic\\
\v{Z}itn\'{a} 25\\
115\,67 Prague 1\\
Czech Republic}
\email{martin@mvath.de}

\thanks{The paper was partially written in the framework of the DFG projects
SFB\,647 and SFB\,910, and partially while the second author stayed at
the Academy of Sciences of the Czech Republic. Financial support by the DFG
and by the Academy of Sciences of the Czech Republic is gratefully
acknowledged.}



\begin{abstract}
An asymptotic stability result for parabolic semilinear problems
in $L_2(\Omega)$ and interpolation spaces is shown. Some known results
about stability in $W^{1,2}(\Omega)$ are improved for semilinear
parabolic mixed boundary value problems. The approach is based on
Amann's power extrapolation scales.
In a Hilbert space setting, a better understanding of this approach
is provided for operators satisfying Kato's square root problem;
as a side result some equivalent characterizations of these operators
are obtained.
\end{abstract}

\maketitle

\section{Introduction}

To the authors knowledge, results dealing with linear stability of
semilinear equations $u_t+Au=f(u)$ always make use of semigroup techniques.
In the simplest of these results for $C_0$-semigroups~\cite{Smoller},
the nonlinearity $f$ is assumed to act (and be e.g.\ differentiable)
in the same Banach space $H$ in which the semigroup acts.
In the case of heat equations or reaction-diffusion systems, i.e., when the
semigroup is (essentially) given by the Laplace operator, the classial
choices of the space $H$ are e.g.\ $W^{1,p}(\Omega)$ (or subspaces taking
some boundary conditions into account) or $L_p(\Omega)$.
However, in these cases, the nonlinearity given by a superposition operator
is differentiable if and only if it is affine, see e.g.~\cite{KrasTop}.

One possible solution of this problem is to work in spaces of continuous
functions, see~\cite{Lunardi}. However, this is not possible if one wants
to consider Sobolev or $L_p$ spaces. In this case, another approach can
be found in~\cite{Henry}, where the nonlinearity
is assumed to act only from a space $H_\alpha$ with $\alpha\in[0,1)$
into $H$ with $H_\alpha$ being the domain of a fractional power
of the (negative of the) generator of the semigroup.
This idea can be extended to somewhat more general interpolation
spaces, which in some cases avoids the problem that the space depends on
the operator (which is important for quasilinear problems), see~\cite{Amann}.
The classical folklore way to apply this result is to work in $H=L_p(\Omega)$,
and one obtains that $H_\alpha$ is for sufficiently large $p$
embedded into $C(\overline\Omega)$, hence differentiability of the
nonlinearity is not an issue anymore. However, one obtains
asymptotic stability only in the space $H_\alpha$ with large $\alpha>0$
since otherwise one ends up with very restrictive (or in case $\alpha=0$ even
degenerate) hypotheses about the nonlinearity $f$.

Results obtained in this way are usually not comparable with
instability results for e.g.\ obstacle problems where one sometimes
obtains instability in the $W^{1,2}(\Omega)$ or $L_2(\Omega)$ topology
by completely different methods. In order to compare the
problems with and without obstacles, we should thus know something about
their linear stability in $W^{1,2}(\Omega)$ and $L_2(\Omega)$.
Now the folklore way to do this is rather suboptimal.
For a stability result for the Laplace operator with
Neumann boundary conditions in the $W^{1,2}(\Omega)$ topology, we would need
to consider $H=L_2(\Omega)$ and get $H_{1/2}=W^{1,2}(\Omega)$, hence our
nonlinearity has to act from $W^{1,2}(\Omega)$ into $L_2(\Omega)$, which
(in space dimension $N>1$) amounts to a certain growth hypotheses on the
function generating the superposition operator; a corresponding result
for a reaction diffusion system was formulated e.g.\
in~\cite{VaethQuittnerProc}. Moreover, to get a stability
result in the $L_2(\Omega)$-topology in this way, one would have to choose
$\alpha=0$, that is, one would need to consider the nonlinearity acting from
$H$ into itself. As mentioned above, this means that one cannot consider
differentiable nonlinearities of superposition type at all.

Note that, in contrast, if one is interested in stationary solutions,
i.e.\ in solutions of the corresponding elliptic problem,
a natural approach is to consider
the superposition operator acting from $W^{1,2}(\Omega)$ into the antidual
space with respect to the $L_2$-scalar product, i.e., into the antidual space
$W^{1,2}(\Omega)'$. Since $L_p(\Omega)\subseteq W^{1,2}(\Omega)'$
for some $p<2$, this approach requires a milder growth condition than
if the nonlinearity acts from $W^{1,2}(\Omega)$ into $L_2(\Omega)$.
It would be nice to have also a corresponding result with weaker
growth hypothesis for the parabolic case.

It is perhaps not so well known that Amann's technique of
power extrapolation spaces can be used to solve both problems simultaneously.
One can obtain results about linear stability in $W^{1,2}(\Omega)$
under the ``natural'' acting conditions as in the elliptic problem
(that is, for subcritical growth of the nonlinearity),
thus relaxing the growth hypothesis supposed in
e.g.~\cite{VaethQuittnerProc}. Moreover, simultaneously, one can obtain
stability result in the $L_2(\Omega)$ topology which is really applicable for
superposition operators.

We note that stable manifolds using extrapolation spaces have also been
introduced in~\cite{FiedlerVishik} to obtain similarly a H\"{o}lder condition
with respect to an averaging parameter. For particular parabolic equations
similar approaches in $L_p(\Omega)$ with $p$ close to~$2$ have
been studied by K. Gr\"{o}ger, J. Rehberg, and others (see
e.g.~\cite{HallerRehberg}, particularly the proof of Lemma~5.3).
The authors thank J. Rehberg for pointing out references to
corresponding abstract results (personal communication).

The purpose of this paper is to carry out this technique, which is not
straightforward, since e.g.\ spectral properties of perturbed
operators do not carry over immediately to ``extrapolated'' operators.
We begin with a Banach space setting and then
concentrate on the case of an operator $A$ generated by a
``strongly accretive'' form in a Hilbert space. For such an operator, one
obtains an abstract extension $\Acl$ in a natural manner.
We will show that $\Acl$ is generated by a ``strongly accretive'' form
if and only if $A$ solves Kato's square root problem.
Moreover, this is the case if and only if $\Acl$ is the
``extrapolated'' operator of $A$ of order $-1/2$, and in this case all
extrapolated\slash interpolated operators (of any negative or positive order)
solve Kato's square root problem, too.

The plan of the paper is as follows. In Section~\ref{s:classical},
we recall (slight extensions of) the classical results related to
stability from~\cite{Henry}. In Section~\ref{s:results}, we extend
these results under milder hypotheses about the nonlinearities
in terms of Amann's extrapolated power scales. The rest of the paper is devoted
to the Hilbert space setting, where the technique is particularly fruitful.
In Section~\ref{s:Hilbert}, we clarify the relation between these extrapolated
power scales, strongly accretive operators, and Kato's square root problem.
Applications to semilinear parabolic problems are given in Section~\ref{s:pde};
in particular, stability of a reaction-diffusion system is obtained,
for which instability is known under obstacles~\cite{VaethQuittner}.
In the appendix, we briefly discuss a sufficient condition for an operator
to solve Kato's square root problem, which follows as a by-result of our
main theorem of Section~\ref{s:Hilbert}.

\section{Summary of Classical Results}\label{s:classical}

Our main interest lies in some dynamical assertions about stability
of equilibria for semilinear parabolic equations, which we formulate now.
We start by summarizing (slight extensions of) well-known
results which can be found in e.g.~\cite{Henry}.

Here and in the following, $(H,\Abs\parameter)$ denotes a complex Banach space,
and $A\colon D(A)\to H$ a (densely defined closed) sectorial operator in $H$
in the sense of~\cite{Henry}, that is, $-A$ generates an
analytic $C_0$-semigroup.
Moreover, we assume that the spectrum of $A$ is disjoint from $(-\infty,0]$.
The latter implies that $A$ is positive (of positive type)
in the sense of~\cite{Triebel} (or~\cite{Amann}), and it is actually no
loss of generality, since it can be arranged by adding a corresponding
multiple of the identity to $A$, if necessary.

Since $A$ is of positive type, one can define fractional power operators
$A^\alpha$, $\alpha\in\C$. Here, we use Komatu's characterization of
fractional power operators~\cite{KomatsuII}, which coincides with that
of~\cite[Section~1.15.1]{Triebel} and that of~\cite{Amann}.
For real $\alpha\ge0$, we denote by $H_\alpha$ the domain of
$D(A^\alpha)\subseteq H$, endowed with the norm
\begin{equation}\label{e:Halpha}
\Norm u_{H_\alpha}\coloneqq\Abs{A^{\alpha}u}\text,
\end{equation}
which is equivalent to the graph norm.

In this section, we fix $\alpha\in[0,1)$; the case $\alpha=0$, that is,
$H_\alpha=H$ is explicitly admissible.

Given a subset $U\subseteq\R\times H_\alpha$ and a function
$f\colon U\to2^H$ (we include multi-valued $f$ for completeness),
we consider the problem
\begin{equation}\label{e:semilinear}
u'(t)+Au(t)\in f(t,u(t))\text.
\end{equation}

\begin{definition}\label{d:strongH}
We call $u\in C([t_0,t_1),H)$ a \emph{strong\slash mild solution}
of~\eqref{e:semilinear} if there is a function $f_0\colon(t_0,t_1)\to H$ with
$f_0\in L_1((t_0,\tau),H)$ for every $\tau<t_1$ such that the following holds
for every $t\in(t_0,t_1)$: $(t,u(t))\in U$, $f_0(t)\in f(t,u(t))$, and
\begin{description}
\item[(strong solution)]
$u'(t)\in H$ exists in the sense of the norm of $H$,
$u(t)\in D(A)$, and $u'(t)+Au(t)=f_0(t)$.
\item[(mild solution)]
\begin{equation}\label{e:mild}
u(t)=e^{-(t-t_0)A}u(t_0)+\int_{t_0}^te^{-(t-s)A}f_0(s)\ds\text.
\end{equation}
\end{description}
\end{definition}

\begin{theorem}[Classical Regularity]\label{t:classreg}
Every strong solution is a mild solution, and the converse holds if $f_0$
in Definition~\ref{d:strongH} is locally H\"{o}lder continuous.
Moreover, if $u\colon[t_0,t_1)\to H$ satisfies~\eqref{e:mild} for all
$t\in(t_0,t_1)$ then
\begin{enumerate}
\item if $f_0\in L_1((t_0,\tau),H)$ for every $\tau\in(t_0,t_1)$ then
$u\in C([t_0,t_1),H)$.
\item if for every $\tau\in(t_0,t_1)$ there is $p>1/(1-\alpha)$ with
$f_0\in L_p((t_0,\tau),H_{-\gamma})$,
then $u\colon(t_0,t_1)\to H_\alpha$ is locally H\"{o}lder continuous, and
$u\in C([t_0,t_1),H_\alpha)$ if and only if $u(t_0)\in H_\alpha$.
\end{enumerate}
\end{theorem}
\begin{proof}
The first assertions can be found as
e.g.~\cite[Corollary~4.2.2]{Pazy}.
The remaining assertions follow by a standard calculation for weakly singular
integrals (see e.g.~\cite[Satz~6.12]{VaethAnalysis} for the scalar case)
by using that $e^{-tA}\colon H\to H_\alpha$ is bounded for $t>0$
by $C_0/t^\alpha$ with $C_0$ independent of $t\ge0$,
that the function $g_{u_0}\colon[0,\infty)\to H_\alpha$,
$g_{u_0}(t)\coloneqq e^{-tA}u_0$ is locally H\"{o}lder continuous
on $(0,\infty)$ if $u_0\in H$ by~\cite[Theorem~2.6.3]{Pazy}, and
continuous at $0$ if $u_0\in H_\alpha$, because for
$u_1\coloneqq A^\alpha u_0$ there holds $A^\alpha g(t)=e^{-tA}u_1$, see
e.g.~\cite[Theorem~2.6.13(b,c)]{Pazy}.
\end{proof}

Concerning existence results, we will for simplicity only consider
single-valued $f$ in which case we also get uniqueness and regularity.
We say that $f$ satisfies a \emph{right local H\"{o}lder-Lipschitz} condition
if for each $(t_0,u_0)\in U$ there is a (relative) neighborhood
$U_0\subseteq[t_0,\infty)\times H_\alpha$ of $(t_0,u_0)$ with $U_0\subseteq U$
such that there are constants $L<\infty$ and $\sigma>0$ with
\begin{equation}\label{e:Lipschitz}
\Abs{f(t,u)-f(s,v)}\le L\cdot(\Abs{t-s}^\sigma+\Norm{u-v}_{H_\alpha})
\quad\text{for all $(t,u),(s,v)\in U_0$.}
\end{equation}
We call $f$ \emph{left-locally bounded into $H$} if for each $t_1>t_0$ and
each bounded $M\subseteq H_\alpha$ there is some $\varepsilon>0$ such that
$f\bigl(U\cap([t_1-\varepsilon,t_1)\times M)\bigr)$ is bounded in $H$.

\begin{theorem}[Classical Uniqueness, Existence, Maximal Interval]%
\label{t:classglobal}
\begin{enumerate}
\item\label{i:classlocal}
If $f\colon U\to H$ satisfies a right local H\"{o}lder-Lipschitz condition,
then for every $(t_0,u_0)\in U$ and $t_1\in(t_0,\infty]$ there is at most
one mild solution $u\in C([t_0,t_1),H_\alpha)$
of~\eqref{e:semilinear} satisfying $u(t_0)=u_0$.
\item
Moreover, such a strong solution exists with some $t_1>t_0$,
and if $f$ is left-locally bounded into $H$, then
some maximal $t_1>t_0$ can be chosen such that either $t_1=\infty$ or
$\Norm{u(t)}_{H_{\alpha}}\to\infty$ as $t\to t_1$ or the limit
$u_1=\lim_{t\to t_1^-}u(t)$ exists in $H_\alpha$ with $(t_1,u_1)\notin U$.
\end{enumerate}
\end{theorem}
\begin{proof}
The result is shown in the proofs of~\cite[Theorems~3.3.3 and~3.3.4]{Henry}.
We recall that local uniqueness implies global uniqueness by
standard arguments.
\end{proof}

Theorem~\ref{t:classglobal} is only the motivation for the subsequent
classical asymptotic stability result.

We formulate this result even for multi-valued $f\colon U\to2^H$, since the
proof is practically the same as in the classical single-valued case.
We call $u_0\in D(A)$ an \emph{equilibrium} of~\eqref{e:semilinear}
if $0\in Au_0+f(t,u_0)$ for all $t>0$ and make the following hypothesis:

\begin{description}
\item[\Hypo{B}] Let $u_0$ be an equilibrium, $U_1\subseteq H_\alpha$ an
open neighborhood of $u_0$ and $[0,\infty)\times U_1\subseteq U$.
Assume that there is a bounded linear map $B\colon H_\alpha\to H$ such that
the function $g(t,u)\coloneqq f(t,u_0+u)+Au_0-Bu$ satisfies
\[\lim_{\Norm u_{H_\alpha}\to0}\,
\frac{\sup\bigl\{\Abs v:v\in g\bigl((0,\infty)\times\Curly u\bigr)\bigr\}}%
{\Norm u_{H_\alpha}}=0\text.\]
(Here, we use the convention $\sup\emptyset\coloneqq0$.)
\end{description}

If $f(t,\parameter)$ is single-valued in a neighborhood of $u_0$,
then $Au_0=-f(t,u_0)$ so that hypothesis~\Hypo{B} means that
$f(t,\parameter)$ is Fr\'echet differentiable at $u_0$ with derivative $B$,
uniformly with respect to $t\in[0,\infty)$. We denote by $\sigma(A-B)$
the spectrum of $A-B$ in $H$.

\begin{theorem}[Classical Asymptotic Stability]\label{t:classass}
Under hypothesis~\Hypo{B}, assume that there is $\lambda_0>0$ such that
$\sigma(A-B)\subseteq\Curly{\lambda\in\C:\Real\lambda>\lambda_0}$.

Then there exist $M_1,M_2>0$ such that if
$t_1>t_0\ge0$ and if $u\in C([t_0,t_1),H_\alpha)$ is a
mild solution of~\eqref{e:semilinear} on $[t_0,t_1)$ with
$\Norm{u(t_0)-u_0}_{H_\alpha}\le M_1$,
then $u$ satisfies the asymptotic stability estimate
\begin{equation}\label{e:asymptest}
\Norm{u(t)-u_0}_{H_\alpha}\le
M_2e^{-\lambda_0(t-t_0)}\Norm{u(t_0)-u_0}_{H_\alpha}
\quad\text{for all $t\in[t_0,t_1)$.}
\end{equation}
If $f$ satisfies in addition the hypotheses of part~\iref{i:classlocal} of
Theorem~\ref{t:classglobal}, then additionally for every $t_0\ge0$
and every $u_1\in H_\alpha$ with $\Norm{u_1-u_0}\le M_1$ a unique
strong solution $u\in C([t_0,\infty),H_\alpha)$ with $u(t_0)=u_1$ exists
and satisfies~\eqref{e:asymptest} with $t_1=\infty$.
\end{theorem}
\begin{proof}
The result is proved analogously to~\cite[Theorem~5.1.1]{Henry}.
\end{proof}

The above classical results have several disadvantages.
In the lack of a local H\"{o}lder-Lipschitz condition or, even more,
in the multi-valued case, there may be solutions of~\eqref{e:semilinear}
in a weaker sense which are not covered in Theorem~\ref{t:classass}.
Moreover, in the most important case $H=L_2(\Omega)$ and when $f$ is
generated by a superposition operator, the choice $\alpha=0$ is not possible,
that is, one cannot obtain a nontrivial stability criterion
in $H_0=L_2(\Omega)$ by Theorem~\ref{t:classass}.
Indeed, it is well known that any differentiable (single-valued)
superposition operator $f$ in $L_2(\Omega)$ is actually affine,
see e.g.~\cite{KrasTop}.

In addition, even just the acting condition $f\colon U\to H$ in the
spaces $H_\alpha=V=W^{1,2}(\Omega)$ and $H=L_2(\Omega)$ leads to a growth
condition on $f$ which appears unnecessarily restrictive. In the study
of stationary solutions, one typically only requires that $f\colon V\to V'$
is continuous (and usually compact) which is satisfied under a much
milder growth condition.

A solution of this problem is to replace the image space $H$ in
Theorems~\ref{t:classglobal} and~\ref{t:classass} by a larger space
with a weaker topology. This can be done using Amann's extrapolated
power scales.

\section{Results Using Extrapolated Power Scales}\label{s:results}

In this section, we make the same general hypotheses about $A$ as in the
previous section, that is, $A$ is a densely defined sectorial operator with
spectrum disjoint from $(-\infty,0]$.
We define the norm~\eqref{e:Halpha} on $H$ also in case $\alpha<0$.
In general, $H$ is not complete with respect to this norm, and so we define
$H_\alpha$ for $\alpha<0$ as the corresponding completion. With this notation,
Amann's extrapolated power scale theory (see~\cite{AmannExtrapol}
or~\cite[Chapter~V]{Amann}) provides the following results.

All embeddings $H_\beta\subseteq H_\alpha$ with $\alpha<\beta$ are
dense; they are all compact if and only if one of these embeddings is compact,
and this is the case if and only if $A$ has a compact resolvent.

For $\alpha\in\R$, $A$ induces by graph closure (or restriction in case
$\alpha\ge0$) isomorphisms $A_\alpha\colon H_{1+\alpha}\to H_\alpha$
(hence $A_\alpha$ is closed as an operator in $H_\alpha$
by~\cite[Lemma~I.1.1.2]{Amann}).  For $\beta>\alpha$, $A_\beta$ is the
$H_\beta$-realization of $A_\alpha$, that is,
$A_\beta=A_\alpha|_{D(A_\beta)}$ with
$D(A_\beta)=A_\alpha^{-1}(H_\beta)=H_{\beta+1}$.
All $A_\alpha$ are thus densely defined operators in $H_\alpha$.
They have the same spectrum as $A$ and are sectorial in $H_\alpha$
(hence of positive type).
In particular, $-A_\alpha$ generates an analytic semigroup in $H_\alpha$.
The corresponding semigroups correspond to each other by restriction
or (unique) continuous extension, respectively.

It is remarkable that for the following result it is sufficient that $A$ is a
densely defined operator in a Banach space $H$ of positive type.
It follows by combining Proposition~V.1.2.6 with Theorem~V.1.3.9
(and their proofs) from~\cite{Amann},
cf.\ e.g.~\cite[Corollary~V.1.3.9]{Amann}.

\begin{lemma}\label{l:Jalpha}
There is a family of isometric isomorphisms
$J_{\alpha,\beta}\colon H_\alpha\to H_\beta$ for $\alpha,\beta\in\R$ with
$A_\alpha=J_{\alpha,\beta}^{-1}A_\beta J_{\alpha+1,\beta+1}$. In fact,
$J_{\alpha,\beta}=(A_\alpha)^{\alpha-\beta}$ for $\alpha\le\beta$, and
$J_{\alpha,\beta}=J_{\beta,\alpha}^{-1}=A_\beta^{\alpha-\beta}$ for
$\alpha\ge\beta$. Moreover, if $\gamma\ge0$ then
$J_{\alpha+\gamma,\beta+\gamma}=J_{\alpha,\beta}|_{H_{\alpha+\gamma}}$
is the $H_{\beta+\gamma}$-realization of $J_{\alpha,\beta}$, that is,
$H_{\alpha+\gamma}=J_{\alpha,\beta}^{-1}(H_{\beta+\gamma})$.
\end{lemma}

\begin{corollary}\label{c:realize}
Let $\alpha\in\R$, $\gamma\ge0$. If $\sigma\le0$, then
$A_{\alpha+\gamma}^\sigma=A_\alpha^\sigma|_{H_{\alpha+\gamma}}$.
If $\sigma\ge0$ then
$A_{\alpha+\gamma}^\sigma=A_\alpha^\sigma|_{H_{\alpha+\gamma+\sigma}}$ is the
$H_{\alpha+\gamma}$-realization of $A_\alpha^\sigma$, that is,
$H_{\alpha+\gamma+\sigma}=(A_\alpha^\sigma)^{-1}(H_{\alpha+\gamma})$.
\end{corollary}

We need to apply Amann's theory in different scales of spaces.
The crucial observation for us is that there is a relation between these
different scales. We already remarked that all our hypotheses
which we assumed for $(H,A)$ are also satisfied with the choice
$(H_{-\gamma},A_{-\gamma})$. Starting with this couple instead, we obtain by
the above definition a corresponding family of spaces $(H_{-\gamma})_\alpha$.
For instance, we have $(H_{-\gamma})_0=H_{-\gamma}$.
The following lemma states that these spaces are related to our original
spaces $H_\alpha$.

\begin{lemma}\label{l:main}
If $\alpha,\gamma\in\R$ then $H_\alpha=(H_{-\gamma})_{\alpha+\gamma}$.
\end{lemma}
\begin{proof}
Set $\beta\coloneqq\alpha+\gamma$.
In case $\beta\ge0$, we obtain from Lemma~\ref{l:Jalpha} that
$A_{-\gamma}^{\beta}=J_{\alpha,-\gamma}$ is norm-preserving
from $H_\alpha$ onto $H_{-\gamma}$. Hence, by the definition of
$(H_{-\gamma})_{\beta}$, we obtain
\[u\in H_\alpha\iff A_{-\gamma}^{\beta}u\in H_{-\gamma}\iff
u\in(H_{-\gamma})_{\beta}\text,\]
and the norm equality
\[\Norm u_{H_\alpha}=\Norm{A_{-\gamma}^{\beta}u}_{H_{-\gamma}}=
\Norm u_{(H_{-\gamma})_{\beta}}\text.\]
In case $\beta\le0$, we obtain from Lemma~\ref{l:Jalpha} that
$A_{\alpha}^\beta=J_{\alpha,-\gamma}$ is norm-preserving from
$H_\alpha$ onto $H_{-\gamma}$. Using Corollary~\ref{c:realize}, we obtain
\[\Norm u_{H_\alpha}=
\Norm{A_{\alpha}^{\beta}u}_{H_{-\gamma}}=
\Norm{A_{-\gamma}^{\beta}u}_{H_{-\gamma}}=
\Norm u_{(H_{-\gamma})_{\beta}}\]
for all $u\in H_{-\gamma}$. Since $H_{-\gamma}$ is densely embedded into
$H_\alpha$ as well as into $(H_{-\gamma})_{\beta}$, the assertion follows.
\end{proof}

Fixing now, throughout this section,
\begin{equation}\label{e:gammaalpha}
\alpha\in[0,1)\text,\qquad\gamma\in[0,1-\alpha)\text,
\end{equation}
we relax the acting condition of $f$ by replacing $H$ by $H_{-\gamma}$ in
the results of Section~\ref{s:classical}, that is, we require now only
$f\colon U\to2^{H_{-\gamma}}$ with $U\subseteq\R\times H_\alpha$.

\begin{definition}\label{d:weak}
We call $u\in C([t_0,t_1),H_{-\gamma})$ a
\emph{$\gamma$-weak\slash mild solution}
of~\eqref{e:semilinear} if there is some $f_0\colon(t_0,t_1)\to H_{-\gamma}$
with $f_0\in L_1((t_0,\tau),H_{-\gamma})$ for every $\tau\in(t_0,t_1)$
such that the following holds for every $t\in(t_0,t_1)$: $(t,u(t))\in U$;
$f_0(t)\in f(t,u(t))$, and
\begin{description}
\item[($\gamma$-weak solution)]
$u'(t)\in H_{-\gamma}$ exists in the sense of the norm of $H_{-\gamma}$,
$u(t)\in D(A_{-\gamma})$, and $u'(t)+A_{-\gamma}u(t)=f_0(t)$.
\item[($\gamma$-mild solution)]
\begin{equation}\label{e:mildgamma}
u(t)=e^{-(t-t_0)A_{-\gamma}}u(t_0)+
\int_{t_0}^te^{-(t-s)A_{-\gamma}}f_0(s)\ds\text.
\end{equation}
\end{description}
\end{definition}

\begin{remark}
Since the semigroups are restrictions of each other, we can
replace~\eqref{e:mildgamma} equivalently by
\begin{equation*}
u(t)=e^{-(t-t_0)A_{-\gamma_0}}u(t_0)+
\int_{t_0}^te^{-(t-s)A_{-\gamma_0}}f_0(s)\ds
\end{equation*}
for every $\gamma_0\ge\gamma$. We point this out, because in the subsequent
Hilbert space setting, the operator $A_{-1/2}$ is ``explicitly'' given, and so
it is natural to choose $\gamma_0=1/2$ in case $\gamma\le1/2$.
\end{remark}

The purpose of relaxing the acting condition of $f$ is that we can
also relax the corresponding continuity hypotheses.
We replace~\eqref{e:Lipschitz} by
\begin{equation}\label{e:goodLip}
\Norm{f(t,u)-f(s,v)}_{H_{-\gamma}}\le
L\cdot(\Abs{t-s}^\sigma+\Norm{u-v}_{H_\alpha})
\quad\text{for all $(t,u)\in U_0$.}
\end{equation}
Similarly, we call $f$ \emph{left-locally bounded into $H_{-\gamma}$}
if for each $t_1>t_0$ and each bounded $M\subseteq H_\alpha$ there is some
$\varepsilon>0$ such that $f\bigl(U\cap([t_1-\varepsilon,t_1)\times M)\bigr)$
is bounded in $H_{-\gamma}$. Then we obtain the following generalization of
Theorem~\ref{t:classreg}.

\begin{theorem}[Regularity]
Every $\gamma$-weak solution is a $\gamma$-mild solution, and the
converse holds if $f_0$ in Definition~\ref{d:weak} is
locally H\"{o}lder continuous. Moreover, if $u\colon[t_0,t_1)\to H_{-\gamma}$
satisfies~\eqref{e:mildgamma} for all $t\in(t_0,t_1)$ then
\begin{enumerate}
\item if $f_0\in L_1((t_0,\tau),H_{-\gamma})$ for every $\tau\in(t_0,t_1)$ then
$u\in C([t_0,t_1),H_{-\gamma})$.
\item if for every $\tau\in(t_0,t_1)$ there is $p>1/(1-\alpha)$ with
$f_0\in L_p((t_0,\tau),H_{-\gamma})$,
then $u\colon(t_0,t_1)\to H_\alpha$ is locally H\"{o}lder continuous, and
$u\in C([t_0,t_1),H_\alpha)$ if and only if $u(t_0)\in H_\alpha$.
\end{enumerate}
\end{theorem}
\begin{proof}
This is essentially Theorem~\ref{t:classreg} with $(H,A,\alpha)$ replaced by
$(H_{-\gamma},A_{-\gamma},\beta)$ with $\beta\coloneqq\alpha+\gamma$.
Note that the semigroup generated by $A_{-\gamma}$ is indeed an extension of
the semigroup generated by $A$. Moreover,
by Lemma~\ref{l:main}, the space $(H_{-\gamma})_\beta$ in the corresponding
assertion of Theorem~\ref{t:classreg} is indeed the same as the space
$H_\alpha$ in the assertion of Theorem~\ref{t:classreg}.
\end{proof}

In exactly the same way, the following result follows from
Theorem~\ref{t:classglobal}.

\begin{theorem}[Uniqueness, Existence, Maximal Interval]\label{t:existence}
Suppose~\eqref{e:gammaalpha}.
\begin{enumerate}
\item\label{i:uniqueness}
If $f\colon U\to H_{-\gamma}$ satisfies a right local H\"{o}lder-Lipschitz
condition in the sense~\eqref{e:goodLip}, then for every $(t_0,u_0)\in U$
and $t_1\in(t_0,\infty]$ there is at most one
$\gamma$-mild solution $u\in C([t_0,t_1),H_\alpha)$
of~\eqref{e:semilinear} satisfying $u(t_0)=u_0$.
\item
Moreover, a $\gamma$-weak solution exists with some $t_1>t_0$, and
if $f$ is left-locally bounded into $H_{-\gamma}$, then
some maximal $t_1>t_0$ can be chosen such that either $t_1=\infty$ or
$\Norm{u(t)}_{H_{\alpha}}\to\infty$ as $t\to t_1$ or the limit
$u_1=\lim_{t\to t_1^-}u(t)$ exists in $H_\alpha$ with $(t_1,u_1)\notin U$.
\end{enumerate}
\end{theorem}

To generalize Theorem~\ref{t:classass}, we note that we assume now
$f\colon U\to2^{H_{-\gamma}}$ so that we have to generalize some notions.

\begin{definition}
An element $u_0\in H_{1-\gamma}$ is called a
\emph{$\gamma$-weak equilibrium} of~\eqref{e:semilinear}
if $A_{-\gamma}u_0\in f(t,u_0)$ for every $t>0$.
\end{definition}

Since the operators are extensions of each other, we have:

\begin{remark}\label{r:weakequi}
If $0\le\widetilde\gamma\le\gamma$, then each
$\widetilde\gamma$-weak equilibrium is a $\gamma$-weak equilibrium.
Conversely, if $u_0$ is a $\gamma$-weak equilibrium with
$A_{-\gamma}u_0\in H_{-\widetilde\gamma}$, that is, if
$u_0\in H_{1-\widetilde\gamma}$, then $u_0$ is a
$\widetilde\gamma$-weak equilibrium.
Moreover, ``$0$-weak equilibrium'' means the same as ``equilibrium''.
In particular, each equilibrium $u_0$ is a $\gamma$-weak equilibrium,
and the converse holds if $A_{-\gamma}u_0\in H_{-0}=H$, that is, if
$u_0\in H_1=D(A)$.
\end{remark}

We will make the following hypothesis:

\begin{description}
\item[\Hypo{B_\gamma}]
Let $u_0$ be a $\gamma$-weak equilibrium, $U_1\subseteq H_\alpha$ an
open neighborhood of $u_0$, and $[0,\infty)\times U_1\subseteq U$.
Assume that there is a bounded linear map $B\colon H_\alpha\to H_{-\gamma}$
such that the function
$g(t,u)\coloneqq f(t,u_0+u)+A_{-\gamma}u_0-Bu$ satisfies
\[\lim_{\Norm u_{H_\alpha}\to0}\,
\frac{\sup\bigl\{\Norm v_{H_{-\gamma}}:
v\in g\bigl((0,\infty)\times\Curly u\bigr)\bigr\}}%
{\Norm u_{H_\alpha}}=0\text.\]
\end{description}

Note that $H_{1-\gamma}\subseteq H_\alpha$, and so
$B|_{H_{1-\gamma}}\colon H_{1-\gamma}\to H_{-\gamma}$ is bounded.

In the following result, we consider
$A_{-\gamma}-B\colon H_{1-\gamma}\to H_{-\gamma}$ as an operator in
$H_{-\gamma}$ with domain $H_{1-\gamma}\subseteq H_{-\gamma}$, and we
denote the spectrum of this operator by $\sigma(A_{-\gamma}-B)$.

\begin{theorem}[Asymptotic Stability]\label{t:stability1}
Assume~\eqref{e:gammaalpha}. Let hypothesis~\Hypo{B_\gamma} be satisfied.
Suppose that there is $\lambda_0>0$ such that
$\sigma(A_{-\gamma}-B)\subseteq\Curly{\lambda\in\C:\Real\lambda>\lambda_0}$.

Then there exist $M_1,M_2>0$ such that if $t_1>t_0\ge0$
and $u\in C([t_0,t_1),H_\alpha)$ is a $\gamma$-mild solution
of~\eqref{e:semilinear} with $\Norm{u(t_0)-u_0}_{H_\alpha}\le M_1$,
then $u$ satisfies the asymptotic stability estimate~\eqref{e:asymptest}.

If $f$ satisfies in addition the hypotheses of part~\iref{i:uniqueness} of
Theorem~\ref{t:existence}, then additionally for every $t_0\ge0$
and every $u_1\in H_\alpha$ with $\Norm{u_1-u_0}\le M_1$ a unique
$\gamma$-weak solution $u\in C([t_0,\infty),H_\alpha)$ with $u(t_0)=u_1$
exists and satisfies~\eqref{e:asymptest} with $t_1=\infty$.
\end{theorem}
\begin{proof}
The result follows by applying Theorem~\ref{t:classass} to the operator
$A_{-\gamma}$ in the space $H_{-\gamma}$.
\end{proof}

A stable manifold result in the spirit of Theorem~\ref{t:stability1}
in, roughly speaking, the case $B=0$ was shown in~\cite{FiedlerVishik}.

Theorem~\ref{t:stability1} is not as convenient as it appears at a first
glance, because the operator $A_{-\gamma}-B$ is rather abstract, in general,
and so its spectrum is hard to estimate. Therefore, we formulate two special
cases in which this spectrum is ``easier'' to calculate.

The first case is described in the following result.
Recall that~\eqref{e:gammaalpha} implies that
\[D(A)=H_1\subseteq H_{1-\gamma}\subseteq H_\alpha=D(B)\subseteq H\text.\]
In particular, under the assumptions of the following result,
$A-B$ is an operator in $H$ with domain $H_1$.
Analogously to Theorem~\ref{t:classass}, we denote its spectrum by
$\sigma(A-B)$.

\begin{theorem}\label{t:spectrum}
Assume~\eqref{e:gammaalpha}. Let hypothesis~\Hypo{B_\gamma} be satisfied.
Suppose that at least one of
\begin{equation}\label{e:BHone}
B(H_{1-\gamma})\subseteq H
\end{equation}
or
\begin{gather}
B(H_1)\subseteq H\text,\label{e:BHzero}\\
A_{-\gamma}u-Bu\in H\implies u\in H_1\label{e:BHreal}
\end{gather}
holds. Then $\sigma(A_{-\gamma}-B)=\sigma(A-B)\ne\C$.
In particular, if $\lambda_0>0$ is such
$\sigma(A-B)\subseteq\Curly{\lambda\in\C:\Real\lambda>\lambda_0}$
then the conclusion of Theorem~\ref{t:stability1} holds with that $\lambda_0$.
\end{theorem}
\begin{proof}
We first note that~\eqref{e:BHone} implies~\eqref{e:BHzero}
and~\eqref{e:BHreal}, because $A\colon H_1\to H$ is the $H$-realization of
$A_{-\gamma}\colon H_{1-\gamma}\to H_{-\gamma}$. Moreover,~\eqref{e:BHzero}
and~\eqref{e:BHreal} are equivalent to the assertion that
$C_H\coloneqq A-B\colon H_1\to H$ is the $H$-realization of
$C\coloneqq A_{-\gamma}-B\colon H_{1-\gamma}\to H_{-\gamma}$.

Putting $\beta\coloneqq\alpha+\gamma\in[0,1)$, we have by Lemma~\ref{l:main}
that $D(A_{-\gamma}^\beta)=(H_{-\gamma})_\beta=H_\alpha$.
Since $B\colon H_\alpha\to H_{-\gamma}$ is bounded, and $A_{-\gamma}$ is
sectorial, it follows that $A_{-\gamma}-B$ is sectorial,
see e.g.~\cite[Remark~3.2]{ArendtPerturb}.

Considering $C$ as an operator in $H_{-\gamma}$ with domain
$D(C)=H_{1-\gamma}$, we find in particular that there is $\mu>0$
such that $\mu I+C$ has a bounded inverse $R$, and
$R(H)\subseteq R(H_{-\gamma})=D(C)\subseteq H$.
Hence~\cite[Lemma~V.1.1.1]{Amann} implies that the spectra of $C$ and of its
$H$-realization $C_H$ coincide.
\end{proof}

The other special case of Theorem~\ref{t:stability1} concerns
``weak'' eigenvalues.

\begin{definition}
We call $\lambda\in\C$ a
\emph{$\gamma$-weak eigenvalue of $A-B$ with eigenvector $u\in H_{1-\gamma}$},
if $\lambda$ is an eigenvalue of $A_{-\gamma}-B$ with eigenvector $u$.
\end{definition}

Analogously to Remark~\ref{r:weakequi}, we obtain:

\begin{remark}\label{r:weakev}
If $0\le\widetilde\gamma\le\gamma$
and $\lambda$ is a $\widetilde\gamma$-weak eigenvalue of $A-B$,
then $\lambda$ is a $\gamma$-weak eigenvalue of $A-B$.

Conversely, if $\lambda$ is a $\gamma$-weak eigenvalue of $A-B$
with eigenvector $u\in H_{1-\gamma}\subseteq H$ (recall that $\gamma\le1$)
satisfying $Bu\in H_{-\widetilde\gamma}$ or $u\in H_{1-\widetilde\gamma}$,
then $\lambda$ is a $\widetilde\gamma$-weak eigenvalue of $A-B$
with eigenvector $u\in H_{1-\widetilde\gamma}$.

Moreover, ``$0$-weak eigenvalue'' means the same as ``eigenvalue''.
In particular, each eigenvalue $\lambda$ of $A-B$ is a
$\gamma$-weak eigenvalue of $A-B$; conversely,
if $\lambda$ is a $\gamma$-weak eigenvalue of $A-B$ with eigenvector
$u\in H_{1-\gamma}$ satisfying $Bu\in H$ or $u\in H_1$, then
$\lambda$ is an eigenvalue of $A-B$ with eigenvector $u\in H_1$.
\end{remark}

Remark~\ref{r:weakev} implies in particular:

\begin{proposition}\label{p:BVH}
If at least one of~\eqref{e:BHone} or~\eqref{e:BHreal}
holds, then $\lambda$ is a $\gamma$-weak eigenvalue of $A-B$ with
eigenspace $E$ if and only if $\lambda$ is an eigenvalue of $A-B$ with the
same eigenspace $E$, and automatically $E\subseteq D(A)=H_1$.
\end{proposition}

Now we are in a position to formulate a variant of Theorem~\ref{t:stability1}
in terms of eigenvalues instead of spectral values.

\begin{theorem}[Asymptotic Stability with Eigenvalues]\label{t:stability2}
Assume that one of the embeddings $H_\beta\to H_\delta$ is compact for
$\beta>\delta$, that is, $A$ has a compact resolvent.
Suppose~\eqref{e:gammaalpha}, and let
hypothesis~\Hypo{B_\gamma} be satisfied. Then $\sigma(A_{-\gamma}-B)$
consists only of the $\gamma$-weak eigenvalues of $A-B$, and the corresponding
eigenspaces are finite-dimensional. More general, $A_{-\gamma}-B-\lambda I$
is a Fredholm operator of index~$0$ in $H_{-\gamma}$ for every $\lambda\in\C$.

In particular, if $\lambda_0>0$ is such that
every $\gamma$-weak eigenvalue $\lambda\in\C$ of $A-B$ satisfies
$\lambda>\lambda_0$, then the conclusion of Theorem~\ref{t:stability1} holds
with that $\lambda_0$.
\end{theorem}
\begin{proof}
Recall that the first hypothesis implies that all of the embeddings
$H_\beta\to H_\delta$ are compact if $\beta>\delta$. In particular,
the embedding $H_{1-\gamma}\to H_\alpha$ is compact.

Since $A_{-\gamma}\colon H_{1-\gamma}\to H_{-\gamma}$ is a Fredholm operator
of index~$0$ (in the space $H_{-\gamma}$), it suffices to show
by~\cite[Theorem~5.26]{Kato} that
$C\coloneqq B+\lambda I\colon H_\alpha\to H_{-\gamma}$
is relatively compact with respect to $A_{-\gamma}$. Thus, let
$u_n$ and $A_{-\gamma}u_n$ be bounded in $H_{-\gamma}$. Then $u_n$ is bounded
in $H_{1-\gamma}$, and thus $u_n$ contains a subsequence convergent in
$H_{\alpha}$. Hence, $Cu_n$ contains a subsequence convergent in $H_{-\gamma}$,
as required.
\end{proof}

\section{Relations to Kato's Square Root Problem}\label{s:Hilbert}

For the rest of the paper, we pass to a Hilbert space setting. We assume that
$(H,\Scalr{\parameter,\parameter},\Abs\parameter)$ is a complex Hilbert space.
We use the notation $E'$ for the antidual space of a space $E$, and for an
operator $B$ in $H$, we denote by $B^*$ the Hilbert space adjoint.

Let $(V,\Norm\parameter)$ be a complex Banach space
which is densely embedded into $H$.
The (Banach space) adjoint of the given embedding $i\colon V\to H$ defines
the embedding $i'\colon H'\to V'$ which has automatically a dense range,
since $i$ is one-to-one. Identifying $H$ with $H'$ and $i'(u)$ with
an element of $V'$, we thus have a Gel'fand triple
\begin{equation}\label{e:Gelfand}
V\subseteq H\subseteq V'\text.
\end{equation}
As customary, we denote the pairing of $V'$ and $V$ also by
$\Scalr{\parameter,\parameter}$ (which on $H\times V\subseteq H\times H$
coincides with the scalar product of $H$ by definition of the adjoint,
so that the notation is actually unique).

Throughout this section, let $a\colon V\times V\to\C$ be a sesquilinear form
on $V$ which is continuous, that is, there is $C\in[0,\infty)$ with
\begin{equation}\label{e:continuous}
\Abs{a(u,v)}\le C\Norm u\Norm v\text,
\end{equation}
and which is strongly accretive in the sense that there is $c>0$ with
\begin{equation}\label{e:strongacc}
\Real a(u,u)\ge c\Norm u^2\quad\text{for all $u\in V$.}
\end{equation}
The hypotheses~\eqref{e:continuous} and~\eqref{e:strongacc} mean that
$u\mapsto(\Real a(u,u)+\Abs u^2)^{1/2}$ defines an equivalent norm on $V$
so that $a$ is a closed form on $H\times H$ with domain $D(a)=V$ in the sense
of~\cite{KatoPowersI,Kato,Ouhabaz}.

\begin{remark}\label{r:bM}
For every $M\ge0$, the sesquilinear form
\begin{equation}\label{e:bM}
b_M(u,v)\coloneqq\frac12\bigl(a(u,v)+\overline{a(v,u)}+M\cdot\Scalr{u,v}\bigr)
\end{equation}
is symmetric, that is, $b_M(v,u)=\overline{b_M(u,v)}$, and $b_M$ satisfies
estimates of the type~\eqref{e:continuous} and~\eqref{e:strongacc}.
Hence, $b_M$ becomes a scalar product on $V$, and the norm induced by this
scalar product is equivalent to the norm on $V$. Thus, a form $a$
satisfying~\eqref{e:continuous} and~\eqref{e:strongacc} exists
if and only if $V$ is (isomorphic to) a Hilbert space.
\end{remark}

We associate with $a$ the linear operator $A\colon D(A)\to H$, defined by the
duality $\Scalr{Au,\parameter}=a(u,\parameter)$, that is, $D(A)$ is the set
of all $u\in V$ for which there is some (uniquely determined) $Au\in H$ with
\[\Scalr{Au,\varphi}=a(u,\varphi)\quad\text{for all $\varphi\in V$.}\]
Besides $A$, the form $a$ also induces a bounded operator $\Acl\colon V\to V'$,
defined by
\begin{equation}\label{e:Acdef}
\Scalr{\Acl u,\varphi}=a(u,\varphi)\quad\text{for all $\varphi\in V$,}
\end{equation}
where now the brace on the left denotes the antidual pairing.

The goal of this section is to answer the following questions (which will
be made precise later on):
\begin{enumerate}
\item Is $\Acl$ also associated to a strongly accretive form?
\item Is it true that $\Acl$ corresponds to $A_{-1/2}$ from
Section~\ref{s:results}?
\end{enumerate}
We will actually see that both answers are equivalent.
They are equivalent to the assertion that $A$ solves
Kato's square root problem. Moreover, we will see that an analogous
equivalence holds not only for $A_{-1/2}$ but actually for
$A_\alpha$ with any $\alpha\in\R$.

We first summarize some well-known facts about $A$, see
e.g.~\cite[Section~VI]{Kato}
and~\cite[Proposition~1.51 and Theorem~1.52]{Ouhabaz}.

\begin{proposition}\label{p:Aassoc}
The operator $A$ is closed and densely defined in $H$. The operators
$A^{-1}\colon H\to H$ and $A^{-1}\colon V\to V$ are bounded,
and $D(A)$ is dense in $V$.
The operator $A\colon D(A)\to H$ is sectorial with spectrum contained in
the open right half-plane, and $-A$ generates an analytic contraction
$C_0$-semigroup in $H$. If $a$ is symmetric, then $A\colon D(A)\to H$
is selfadjoint in $H$.
\end{proposition}

Analogous assertions hold also for $\Acl$,
see e.g.~\cite[Theorem~1.55 and subsequent remarks]{Ouhabaz}.

\begin{proposition}
The operator $\Acl$ is an isomorphism of $V$ onto $V'$. It is a
densely defined sectorial operator in $V'$ with spectrum in the
open right half-plane.
\end{proposition}

In view of Proposition~\ref{p:Aassoc}, we are in the setting of
Section~\ref{s:results}, and so we can introduce the operators $A^\alpha$,
the spaces $H_\alpha$, and $A_\alpha\colon H_{\alpha+1}\to H_\alpha$ as
in that section. Moreover, also the adjoint operator $A^*$ is of
positive type, and $(A^{\alpha})^*=(A^*)^{\alpha}$,
see~\cite[Lemma~V.1.4.11]{Amann}.
Let $\Hcc{\alpha}$ denote the spaces of Section~\ref{s:results} generated by
the operator $A^*$ in place of $A$.

We consider also $\Acl$ as an unbounded operator in $V'$
with domain $D(\Acl)=V$. Also $\Acl$ is of positive type, and thus
$\Acl^{\alpha}$ are defined for all $\alpha\in\C$; for $\alpha\ge0$,
we endow $D(\Acl^\alpha)$ with the norms
$\Norm u_{D(\Acl^\alpha)}=\Norm{\Acl^\alpha u}_{V'}$ for $\alpha\ge0$ so that
$\Acl^\alpha\colon D(\Acl)\to V'$ are isometric isometries.

Let us first note that the spaces $H_\alpha$ have a more convenient
characterization in our setting. To this end, we denote by
$[\parameter,\parameter]_{\theta}$ the complex interpolation functor of order
$\theta$, see e.g.~\cite{Triebel}.

\begin{proposition}\label{p:powers}
For every $s\in\R$ the operator $A^{is}$ is bounded in $H$ with
$\Norm{A^{is}}\le e^{\pi\Abs s/2}$, and we have the reiteration formulas
\begin{equation}\label{e:reiterate}
H_{(1-\theta)\alpha+\theta\beta}\cong[H_\alpha,H_\beta]_{\theta}
\quad\text{if $\alpha,\beta\in\R$, $0<\theta<1$,}
\end{equation}
and the duality representation
\begin{equation}\label{e:dual}
H_{-\gamma}\cong\Hcb{\gamma}\qquad\text{if $-1\le\gamma\le1$.}
\end{equation}
Additionally,
\begin{equation}\label{e:Katodomw}
\Hcc{\gamma}\cong H_\gamma\quad\text{if $\gamma\in[0,1/2)$.}
\end{equation}
If $a$ is symmetric, then $\Hcc{1/2}=H_{1/2}\cong V$.
\end{proposition}
\begin{proof}
The norm estimate $\Norm{A^{is}}\le e^{\pi\Abs s/2}$ is shown
in~\cite[Theorem~5]{KatoPowersII}. Using this, we obtain
the reiteration formula~\eqref{e:reiterate} from~\cite[Theorem~V.1.5.4]{Amann}.
The duality formula~\eqref{e:dual} can be found
as~\cite[Proposition~V.1.5.5]{Amann} (note~\cite[Remark~V.1.5.16]{Amann}).
The assertion~\eqref{e:Katodomw} follows from~\cite[Theorem~1.1]{KatoPowersI}.
For symmetric $a$, we have $A=A^*$, hence $H_{1/2}=\Hcc{1/2}$,
and the assertion $H_{1/2}\cong V$ is shown in~\cite{KatoPowersI}
(see also~\cite[Theorem~8.1]{Ouhabaz}).
\end{proof}

The last assertion of Proposition~\ref{p:powers} suggests the following
definition.

\begin{definition}\label{d:Kato}
We call an operator $A$ in $H$ a \emph{Kato operator} (with a form on $V$)
if it is the operator associated with a sesquilinear form
$a\colon V\times V\to\C$ satisfying~\eqref{e:continuous}
and~\eqref{e:strongacc} and $H_{1/2}\cong V$.
\end{definition}

Recall that we assumed throughout that $A$ is associated with a form $a$
on $V$ satisfying~\eqref{e:continuous} and~\eqref{e:strongacc}. Thus, $A$
is a Kato operator if and only if $H_{1/2}\cong V$.

\begin{proposition}
$A$ is a Kato operator if and only if $\Hcc{1/2}\cong H_{1/2}$.
In particular, $A$ is a Kato operator if and only if
\begin{equation}\label{e:Katodom}
\Hcc{\gamma}\cong H_\gamma\quad\text{for all $\gamma\in[0,1/2]$.}
\end{equation}
$A$ is a Kato operator if and only if $A^*$ is a Kato operator.
\end{proposition}
\begin{proof}
The first assertion is a special case of~\cite[Theorem~1]{KatoPowersII},
and~\eqref{e:Katodom} follows in view of~\eqref{e:Katodomw}. The
last assertion follows from the previous assertion and the observation
that $A^*$ is generated by the form $a^*(u,v)\coloneqq\overline{a(v,u)}$,
see e.g.~\cite[Proposition~1.24]{Ouhabaz}.
\end{proof}

The name in Definition~\ref{d:Kato}
is motivated by Kato's famous square root problem
originally posed in~\cite{KatoPowersI}:
to characterize the forms $a$ for which $A$ is a Kato operator.
According to Proposition~\ref{p:powers}, $A$ is a Kato operator
if $a$ is symmetric.
However, also many elliptic differential operators (even nonsymmetric) induce
Kato operators, see e.g.~\cite[Chapter~8]{Ouhabaz} and the references therein
as well as~\cite{Agranovich,EgertKatoMixed,Shamin}.
So the requirement that $A$ is a Kato operator is rather mild from
the viewpoint of applications we have in mind.

The reason why Kato operators are so useful to us in the Hilbert space setting
is the following. In applications to partial differential equations
the form $a$ and the space $V$ are usually explicitly given,
while the spaces $H_\alpha$ and the operators $A_\alpha$
are known only implicitly; in fact, usually even $D(A)=H_1$ is known only
implicitly. For Kato operators, we understand these auxiliary spaces
and operators in case $\alpha\ge-1/2$ much better as the following simple
observation shows.

\begin{proposition}\label{p:Ascript}
If (and only if) $A$ is a Kato operator, we have
\[V\cong H_{1/2}\text,\quad V'\cong\Hcb{1/2}\cong H_{1/2}'\cong H_{-1/2}\text,
\quad A_{-1/2}=\Acl\text.\]
Moreover, in this case, if $\gamma\in(0,1/2)$, then
\begin{equation}\label{e:Hgammabri}
H_{-\gamma}\cong\Hcb{\gamma}\cong
H_\gamma'\cong[H,V]_{2\gamma}'\cong[V',V]_{\frac12+\gamma}'\cong
[V',V]_{\frac12-\gamma}\cong[V',H]_{1-2\gamma}\text,
\end{equation}
$H=H_0\cong[V',V]_{1/2}$, and if $\gamma\in(1/2,1)$ then
\begin{equation}\label{e:Hinthalf}
H_{-\gamma}\cong\Hcb{\gamma}\cong
[D(A^*)',V']_{2-2\gamma}\cong[V,D(A^*)]_{2\gamma-1}'\text.
\end{equation}
Indepently of whether $A$ is a Kato operator, there holds
\begin{equation}\label{e:interpol}
H_{-\gamma}\cong\Hcb{\gamma}\cong[H,D(A^*)']_{\gamma}\cong[H,D(A^*)]_{\gamma}'
\quad\text{for all $\gamma\in(0,1)$.}
\end{equation}
\end{proposition}
\begin{proof}
The first assertion is the definition of a Kato operator, the second follows
by using~\eqref{e:Katodom} and~\eqref{e:dual}.
The identity $A_{-1/2}=\Acl$ (which has to
be interpreted in terms of the canonical isomorphisms, of course),
follows from the density of $D(A)$ in $V$ (Proposition~\ref{p:Aassoc}),
since both operators are bounded from $V$ into $V'$ and coincide
on $D(A)$ with $A$. The formula~\eqref{e:interpol} is shown in a
straightforward manner with~\eqref{e:reiterate} and~\eqref{e:dual}
by inserting $D(A^*)=\Hcc{1}$ and $H=H_0=H_0'=\Hcc{0}=\Hcb{0}$. The
formulas~\eqref{e:Hgammabri}, $H\cong[V',V]_{1/2}$, and~\eqref{e:Hinthalf}
are shown similarly, by using also~\eqref{e:Katodom} and
$V\cong H_{1/2}\cong\Hcc{1/2}$.
\end{proof}

\begin{corollary}\label{c:VbriV}
The last three equality signs in~\eqref{e:Hgammabri} and $[V',V]_{1/2}\cong H$
are valid even if $A$ fails to be a Kato operator.
\end{corollary}
\begin{proof}
The claimed equalities are actually independent of $A$ and $a$. Hence, in
view of Remark~\ref{r:bM}, we can assume that $a$ is symmetric, and in this
case Proposition~\ref{p:powers} implies that the associated
self-adjoint operator $A$ is a Kato operator, and the equalities follow
from Proposition~\ref{p:Ascript}.
\end{proof}

\begin{remark}\label{r:BVH}
Let $\gamma\in[0,1/2]$ and $A$ be a Kato operator.
Then $\lambda$ is a $\gamma$-weak eigenvalue of $A-B$
with corresponding eigenvector $u\ne0$ if and only if $u\in H_{1-\gamma}$ and
\[a(u,\varphi)-\Scalr{Bu,\varphi}=\lambda\Scalr{u,\varphi}\quad
\text{for all $\varphi\in V$.}\]
Moreover, the hypothesis of Proposition~\ref{p:BVH} is satisfied if
$B(V)\subseteq H$.

Indeed, the first assertion follows from the fact that $A_{-\gamma}$ is a
restriction of $A_{-1/2}=\Acl$. The second assertion follows from
$H_{1-\gamma}\subseteq H_{1/2}\cong V$.
\end{remark}

Now we consider the question whether the operator $\Acl$ is associated with
a strongly accretive continuous sesquilinear form $\acl$ on $H$.
We first have to equip $V'$ with an appropriate scalar product.
Our idea for this is to fix a scalar product $b$ on $V$
which generates a norm $\Norm u_b\coloneqq\sqrt{b(u,u)}$ on $V$
equivalent to $\Norm\parameter$.
Note that $\Acl\colon V\to V'$ is an isomorphism, and so
$\Norm u_{a,b}\coloneqq\Norm{\Acl^{-1}u}_b$ defines an
equivalent norm in $V'$. We denote by $X^{a,b}$ the Hilbert space which
we obtain from $V'$ when we pass to this equivalent norm
induced by the scalar product
\begin{equation}\label{e:Scaldef}
\Scalb{u,v}_{a,b}\coloneqq b(\Acl^{-1}u,\Acl^{-1}v)\quad
\text{for all $u,v\in X^{a,b}$.}
\end{equation}
Note that for any choice of $(a,b)$ as above we have $X^{a,b}\cong V'$.
However, it is crucial for our approach to distinguish the various scalar
products.

We emphasize that~\eqref{e:Scaldef} is actually the general form of a scalar
product on $V'$ generating an equivalent norm.
Indeed, if $\Scalb{\parameter,\parameter}_{V'}^*$
is any scalar product on $V'$ such that the generated norm
$\Norm u_{V'}^*\coloneqq\sqrt{\Scalb{u,u}_{V'}^*}$ is equivalent to
$\Norm\parameter_{V'}$, then $\Norm u_V^*\coloneqq\Norm{\Acl u}_{V'}^*$ is a
norm on $V$ which is equivalent to $\Norm\parameter$. Moreover, this norm
is generated by the scalar product
$b(u,v)\coloneqq\Scalb{\Acl u,\Acl v}_{V'}^*$,
and the corresponding scalar product~\eqref{e:Scaldef} is just the
scalar product $\Scalb{\parameter,\parameter}_{V'}^*$ we started with.

The following result characterizes those scalar products $b$ on $V$
(or, equivalently, those scalar products on $V'$) generating an equivalent norm
for which a ``strongly accretive'' continuous form $\acl$ on $H$ exists
such that $\Acl$ is associated to $\acl$.

\begin{proposition}\label{p:a0}
The following assertions are equivalent for every $c_1>0$.
\begin{enumerate}
\item\label{i:a0ex}
There exists a sesquilinear form
$\acl\colon H\times H\to\C$ such that there are constants $c_2,c_3\ge0$ with
\begin{equation}\label{e:a0c1c2}
\Real\acl(u,u)\ge c_1\Abs u^2\text,\quad
\Abs{\acl(u,v)}\le c_2\Abs u\Abs v\text,\quad
\Abs{\acl(u,u)}\le c_3\Abs u^2
\end{equation}
for all $u,v\in H$, and
\begin{equation}\label{e:a0def}
\acl(u,v)=\Scalb{\Acl u,v}_{a,b}\quad\text{for all $u\in V$, $v\in H$.}
\end{equation}
\item\label{i:bbdd}
There is $c_2\ge0$ with
\begin{equation}\label{e:bc1c2}
\Real b(u,A^{-1}u)\ge c_1\Abs u^2\text{ and }
\Abs{b(u,A^{-1}v)}\le c_2\Abs u\Abs v\quad\text{for all $u,v\in D(A)$.}
\end{equation}
\item\label{i:bsbdd}
There is $c_3\ge0$ with
\begin{equation}\label{e:bc1c3}
\Real b(u,A^{-1}u)\ge c_1\Abs u^2\text{ and }
\Abs{b(u,A^{-1}u)}\le c_3\Abs u^2\quad\text{for all $u\in D(A)$.}
\end{equation}
\end{enumerate}
One can equivalently replace $D(A)$ by $V$ in~\iref{i:bbdd}
and~\iref{i:bsbdd}, and the smallest possible constants $c_2,c_3$
in the above assertions are respectively the same.
If~\iref{i:a0ex} holds, then $\acl$
is the unique continuous function $\acl\colon H\times H\to\C$
satisfying~\eqref{e:a0def}, and $\Acl$ is the operator associate to $\acl$
in the Hilbert space $X^{a,b}$, that is, $u\in H$ belongs to $D(\Acl)\cong V$
if and only if there is some $w\in X^{a,b}$ with
$\acl(u,\varphi)=\Scalb{w,\varphi}_{a,b}$ for all $\varphi\in H$.
\end{proposition}
\begin{proof}
``\iref{i:a0ex}$\implies$\iref{i:bbdd}$\implies
$\iref{i:bsbdd}$\implies$\iref{i:bbdd}'':
Let $\acl_1\colon V\times V\to\C$ be the sesqulinear form defined by
$\acl_1(u,v)=\Scalb{\Acl u,v}_{a,b}$. Due to~\eqref{e:a0def}, we have
$\acl|_{V\times V}=\acl_1$. Since $V$ is dense in $H$, $\acl$ is uniquely
determined by its restriction to $V\times V$, and moreover, there is at most
one continuous function $\acl\colon H\times H\to\C$ with
$\acl|_{V\times V}=\acl_1$.

The operator $L$ associated with the form $\acl$ is the
Friedrichs extension of $\Acl$. Since $\Acl$ is sectorial in the
sense of~\cite{Henry}, it follows from~\cite[Theorem~VI.2.9]{Kato}
that $L=\Acl$.

Using~\eqref{e:Scaldef}, the fact that $A$ is the $H$-realization of $\Acl$,
and~\eqref{e:Acdef}, we calculate
\begin{equation}\label{e:a1equal}
\acl_1(u,v)=b(u,\Acl^{-1}v)=b(u,A^{-1}v)\quad\text{for all $u,v\in V$.}
\end{equation}
Hence,~\eqref{e:a0c1c2} implies~\eqref{e:bc1c2} and~\eqref{e:bc1c3}
even for all $u,v\in V$ (with the same constants $c_2,c_3$).
Clearly,~\eqref{e:bc1c2} implies~\eqref{e:bc1c3} with $c_3\coloneqq c_2$.
Conversely, if~\eqref{e:bc1c3} holds, then $\Abs{\acl_1(u,u)}\le c_3\Abs u^2$
for all $u\in D(A)$, and an application of the
polarization identity~\cite[(VI.1.1)]{Kato} for the sesquilinear form
$\acl_1$ shows that $\acl_1$ is bounded with respect to the norm
$\Abs\parameter$. This means that~\eqref{e:bc1c2} holds with some $c_2\ge0$.

``\iref{i:bbdd}$\implies$\iref{i:a0ex}'':
If~\eqref{e:bc1c2} holds then, since $D(A)$ is dense in $V$ and
since the left-hand side is continuous in $V$, we obtain that~\eqref{e:bc1c2}
holds even for all $u,v\in V$. (With the same argument also~\eqref{e:bc1c3}
holds for all $u,v\in V$.) Hence, by~\eqref{e:a1equal},
\[\Real\acl_1(u,u)\ge c_1\Abs u^2\quad\text{and}\quad
\Abs{\acl_1(u,v)}\le c_2\Abs u\Abs v\quad\text{for all $u,v\in V$.}\]
Since $V$ is dense in $H$, it follows that the sesqulinear form $\acl_1$ has a
continuous extension $\acl\colon H\times H\to\C$, and $\acl$
satisfies~\eqref{e:a0c1c2} (with the same constants $c_1,c_2>0$).
\end{proof}

Proposition~\ref{p:a0} motivates the following definition.

\begin{definition}
We call a scalar product $b$ on $V$ an \emph{$A$-Kato scalar product}
if the norm induced by $b$ is equivalent to $\Norm\parameter$, and
if~\eqref{e:bc1c2} or, equivalently,~\eqref{e:bc1c3} (or any of these
with $D(A)$ replaced by $V$) are satisfied with some $c_1>0$ and $c_2,c_3\ge0$.
\end{definition}

We will see that the existence of an $A$-Kato scalar product $b$ on $V$
is actually equivalent to the assertion that $A$ is a Kato operator with
a form on $V$, but at the moment we cannot use this.

\begin{corollary}\label{c:a0symm}
If $a$ is symmetric, then $b\coloneqq a$ is an $A$-Kato scalar product on $V$
for which $\acl$ in Proposition~\ref{p:a0} is given by
$\acl(u,v)=\Scalr{u,v}$ in $H$.
\end{corollary}
\begin{proof}
It suffices to prove that assertion~\iref{i:a0ex} of Proposition~\ref{p:a0}
holds. We calculate for $\acl_1$ from the proof of Proposition~\ref{p:a0}
for all $u,v\in V$, using~\eqref{e:Scaldef} and~\eqref{e:Acdef}, that
\[\overline{\acl_1(u,v)}=\Scalb{v,\Acl u}_{a,a}=
a(\Acl^{-1}v,\Acl^{-1}\Acl u)=\Scalr{\Acl\Acl^{-1}v,u}=
\overline{\Scalr{u,v}}\text.\]
Hence, $\acl_1(u,v)=\Scalr{u,v}$ for all $u,v\in V$, and by continuity
and density $\acl_1$ has a unique continuous extension $\acl(u,v)=\Scalr{u,v}$
for all $u,v\in H$.
\end{proof}

\begin{remark}
Corollary~\ref{c:a0symm} may appear rather surprising.
The form $\acl$ generating $\Acl$ (and thus also defining $A$) is actually
independent of $a$ (and thus of $A$)!
However, the explanation for this apparent contradiction is that the scalar
product in $X^{a,a}\cong V'$ heavily depends on $a$, of course. Although the
generated norms are equivalent, they are not the same, in general.
\end{remark}

Note that the definition of an $A$-Kato scalar product does not make use
of any sesquilinear form.
Hence, it is remarkable that the following result makes sense
and holds for every densely defined operator $A$ of
positive type in a Banach space $H$.

\begin{lemma}\label{l:lowerKato}
For $\alpha,\beta\in\R$, let
$J=J_{\beta+1/2,\alpha+1/2}\colon H_{\beta+1/2}\to H_{\alpha+1/2}$
be the isometric isomorphism of Lemma~\ref{l:Jalpha}.
If $b$ is an $A_\alpha$-Kato scalar product on $H_{\alpha+1/2}$, then
\[B(u,v)\coloneqq b(Ju,Jv)\]
defines an $A_\beta$-Kato scalar product on $H_{\beta+1/2}$ with the same
constants in the respective inequalities of~\eqref{e:bc1c2}
and~\eqref{e:bc1c3} and with the same equivalence constants for the norm.
\end{lemma}
\begin{proof}
By Lemma~\ref{l:Jalpha}, $J$ is the $H_{\alpha+1/2}$-realization of the
isometric isomorphism $J_{\beta,\alpha}\colon H_\beta\to H_\alpha$,
$J_{\beta+1,\alpha+1}$ is a restriction of $J$, and
$A_\beta=J_{\beta,\alpha}^{-1}A_\alpha J_{\beta+1,\alpha+1}$.
Hence,
$JA_\beta^{-1}=J_{\beta+1,\alpha+1}A_\beta^{-1}=A_\alpha^{-1}J_{\beta,\alpha}$.
We obtain
\[B(u,A_\beta^{-1}v)=b(Ju,A_\alpha^{-1}(Jv))\quad\text{and}\quad
\Norm{Ju}_{H_{\alpha+1/2}}=\Norm u_{H_{\beta+1/2}}\]
for all $u,v\in H_{\beta+1/2}$ from which the assertion
follows straightforwardly.
\end{proof}

\begin{lemma}\label{l:induced}
Let $\alpha,\beta\in\R$ and $J=J_{\beta,\alpha}\coloneqq H_\beta\to H_\alpha$
denote the isometric isomorphism of Lemma~\ref{l:Jalpha}.
Let $b_\alpha$ be a scalar product on $H_\alpha$ generating an equivalent norm
and $a_\alpha$ a sesquilinear form on $H_{\alpha+1/2}$ satisfying
\begin{equation*}
\Abs{a_\alpha(u,v)}\le
C\Norm u_{H_{\alpha+1/2}}\Norm v_{H_{\alpha+1/2}}\text,\quad
\Real a_\alpha(u,u)\ge c\Norm u_{H_{\alpha+1/2}}^2\quad
\text{for all $u,v\in H_{\alpha+1/2}$.}
\end{equation*}
Then
\[b_\beta(u,v)\coloneqq b_\alpha(Ju,Jv)\text,\qquad
a_\beta(u,v)\coloneqq a_\alpha(Ju,Jv)\]
define a scalar product on $H_\beta$ generating an equivalent norm with the
same equivalence constants as for $b_\alpha$ and a sesquilinear form on
$H_{\beta+1/2}$, respectively, such that
\begin{equation*}
\Abs{a_\beta(u,v)}\le C\Norm u_{H_{\beta+1/2}}\Norm v_{H_{\beta+1/2}}\text,
\quad
\Real a_\beta(u,u)\ge c\Norm u_{H_{\beta+1/2}}^2\quad
\text{for all $u,v\in H_{\beta+1/2}$.}
\end{equation*}
Moreover, if $A_\alpha$ is the operator associated to $a_\alpha$
with the scalar product $b_\alpha$, then $A_\beta$ is the
operator associated to $a_\beta$ with the scalar product $b_\beta$.
\end{lemma}
\begin{proof}
Note that \(J_{1/2}\coloneqq J_{\beta+1/2,\alpha+1/2}\colon
H_{\beta+1/2}\to H_{\alpha+1/2}\) and
$J_1\coloneqq J_{\beta+1,\alpha+1}\colon H_{\beta+1}\to H_{\alpha+1}$
are norm preserving isomorphisms and the corresponding
$H_{\alpha+1/2}$-realization and $H_{\alpha+1}$-realization of $J$,
respectively, and $A_\beta=J^{-1}A_\alpha J_1$. In particular,
\begin{equation}\label{e:Dinduced}
\begin{gathered}
D(A_\beta)=H_{\beta+1}=J^{-1}(H_{\alpha+1})=J^{-1}(D(A_\alpha))\text,\quad
JA_\beta u=A_\alpha Ju\quad(u\in D(A_\beta))\\
H_{\beta+1/2}=J^{-1}(H_{\alpha+1/2})\text,\quad
\Norm{Ju}_{H_{\alpha+1/2}}=\Norm u_{H_{\beta+1/2}}\quad
(u\in H_{\beta+1/2})\text.
\end{gathered}
\end{equation}
Hence,
$b_\beta(A_\beta u,v)=b_\alpha(A_\alpha Ju,Jv)=a_\alpha(Ju,Jv)=a_\beta(u,v)$
for all $u\in D(A_\beta)$, $v\in H_{\beta+1/2}$. Using this
and~\eqref{e:Dinduced}, the assertions follow straightforwardly.
\end{proof}

Now we can formulate the main result of this section.
It states that if $A$ is a Kato operator, then also each of the operators
$A_\alpha$ is a Kato operator if $H_\alpha$ is equipped with an
appropriate scalar product. As a by-result, we obtain some equivalent
characterizations of Kato operators, in particular the previously mentioned
observation that Kato operators are characterized by the existence of an
$A$-Kato scalar product.

\begin{theorem}\label{t:Kato}
Assume the general hypotheses of this section, that is, $A$ is associated
to a sesquilinear form $a$ on $V$ satisfying~\eqref{e:continuous}
and~\eqref{e:strongacc}. Then the following assertions are equivalent:
\begin{enumerate}
\item\label{i:AKato}
$A$ is a Kato operator with a form on $V$, that is, $H_{1/2}\cong V$.
\item\label{i:arendtiso}
$A^{1/2}$ is an isomorphism of $V$ onto $H$.
\item\label{i:Aprod}
There is an $A$-Kato scalar product on $V$.
\item\label{i:Acaccretive}
There is a scalar product on $V'$ generating an equivalent norm
and a sesquilinear form $\acl\colon H\times H\to\C$
satisfying~\eqref{e:a0c1c2} with $c_1>0$ and $c_2,c_3\ge0$
such that $\Acl$ is associated to $\acl$.
\item\label{i:Aimag}
There are $\varepsilon>0$ and $M\ge0$ such that for all
$s\in(-\varepsilon,\varepsilon)$ the operator $\Acl^{is}$ is bounded in $V'$
with $\Norm{\Acl^{is}}\le M$.
\item\label{i:domain}
$D(\Acl^{1/2})\cong H$.
\item\label{i:Achalf}
$\Acl=A_{-1/2}$.
\item\label{i:AcKato}
There is a scalar product on $V'$ generating an equivalent norm
such that $\Acl$ is a Kato operator with a form on $H$.
\item\label{i:some}
For some $\alpha\in\R$ the space $H_\alpha$ can be equipped with a scalar
product generating an equivalent norm such that the operator
$A_\alpha$ is a Kato operator with a form on $H_{\alpha+1/2}$.
\item\label{i:all}
For every $\alpha\in\R$ the space $H_\alpha$ can be equipped with a scalar
product generating an equivalent norm such that the operator
$A_\alpha$ is a Kato operator with a form on $H_{\alpha+1/2}$.
\end{enumerate}
In each case, $b(u,v)\coloneqq\Scalr{A^{1/2}u,A^{1/2}v}$ defines
an $A$-Kato scalar product on $V$, and  there is $M>0$ such that
for all $s\in\R$ the operator $\Acl^{is}$ is bounded in $V'$ with
$\Norm{\Acl^{is}}\le Me^{\pi\Abs s/2}$. One can even choose $M=1$
if one equips $V'$ with the equivalent norm from~\iref{i:Acaccretive}.
\end{theorem}
\begin{proof}
``\iref{i:AKato}$\iff$\iref{i:arendtiso}'' follows from the fact that
$A^{1/2}$ is an isomorphism of $H_{1/2}$ onto $H_0=H$.

``\iref{i:AKato},\iref{i:arendtiso}$\implies$\iref{i:Aprod}'':
We show that $b(u,v)\coloneqq\Scalr{A^{1/2}u,A^{1/2}v}$ is an
$A$-Kato scalar product on $V$.
By hypothesis, there are constants $C_1,C_2>0$ with
\[C_1\Abs{A^{1/2}v}\le\Norm v\le C_2\Abs{A^{1/2}v}\quad
\text{for all $v\in V\cong H_{1/2}$.}\]
Hence, in this case we calculate for every $u\in V$, noting that $u\in H_{1/2}$
and thus $v\coloneqq A^{-1/2}u\in H_1$, that
\[b(u,A^{-1}u)=\Scalr{A^{1/2}u,A^{1/2}A^{-1}u}=\Scalr{Av,v}=a(v,v)\]
and $\Abs{A^{1/2}v}=\Abs u$. This implies on the one hand that
\[\Real b(u,A^{-1}u)\ge c\Norm v\ge C_1c\Abs{A^{1/2}v}=C_1c\Abs u\text,\]
and on the other hand
\[\Abs{b(u,A^{-1}u)}\le C\Norm v^2\le C_2^2C\Abs{A^{1/2}v}=
C_2^2C\Abs u^2\text.\]
Hence,~\eqref{e:bc1c3} holds with $c_1\coloneqq C_1c$ and
$c_3\coloneqq C_2^2C$.

``\iref{i:Aprod}$\iff$\iref{i:Acaccretive}'' is the content of
Proposition~\ref{p:a0} and its preceding remarks.

``\iref{i:Acaccretive}$\implies$\iref{i:Aimag}'': Equipping the space $V'$
with the scalar product and corresponding norm from~\iref{i:Acaccretive},
we obtain even the uniform estimate $\Norm{\Acl^{is}}\le e^{\pi\Abs s/2}$
for all $s\in\R$ by applying Proposition~\ref{p:powers} with
$(A,a,H)$ replaced by $(\Acl,\acl,V')$.

``\iref{i:Aimag}$\implies$\iref{i:domain}'':
Applying~\cite[Theorem~1.15.3]{Triebel} with the operator $\Acl$ in $V'$,
we find $D(\Acl^{1/2})\cong[D(\Acl^0),D(\Acl^1)]_{1/2}=[V',V]_{1/2}$.
Hence,~\iref{i:domain} follows from Corollary~\ref{c:VbriV}.

``\iref{i:domain}$\implies$\iref{i:Achalf}'' was shown in the proof of
Proposition~\ref{p:Ascript}.

``\iref{i:Achalf}$\implies$\iref{i:AKato}'' follows from the
calculation $H_{1/2}=D(A_{-1/2})=D(\Acl)\cong V$.

``\iref{i:Acaccretive},\iref{i:domain}$\iff$\iref{i:AcKato}'' is our
definition of a Kato operator.

So far, we have shown the equivalences of all assertions of
Theorem~\ref{t:Kato} with the exception of~\iref{i:some} and~\iref{i:all}.
Since ``\iref{i:all}$\implies$\iref{i:AKato}'' and
``\iref{i:AKato}$\implies$\iref{i:some}'' are obvious, we now have to show
``\iref{i:some}$\implies$\iref{i:all}''. Thus, let $\alpha$ denote the number
for which the assertion~\iref{i:some} holds, and let $\beta\in\R$ be arbitrary.
By our choice of $\alpha$, the general hypotheses of our section
hold with $(H_\alpha,A_\alpha)$ in place of $(H,A)$, and by
Lemma~\ref{l:induced}, the same is true for $(H_\beta,A_\beta)$.
Hence, using the already shown implications
``\iref{i:AKato}$\implies$\iref{i:Aprod}'' with $(H_\alpha,A_\alpha)$
and ``\iref{i:Aprod}$\implies$\iref{i:AKato}'' with $(H_\beta,A_\beta)$,
respectively, we obtain in view of Lemma~\ref{l:lowerKato}
that $A_\beta$ is a Kato operator.
\end{proof}

Some of the direct equivalences of Theorem~\ref{t:Kato} are
contained implicitly in~\cite[Section~5.5.2]{ArendtSurvey}.
For instance, the equivalent characterization~\iref{i:domain}
of Kato operators was obtained from these results and applied
in~\cite{Agranovich}. However, it seems that the idea to employ appropriate
scalar products on $V$ appears to be new.

\section{Applications to Semilinear Parabolic PDEs}\label{s:pde}

\subsection{Superposition Operators in $L_2$ and Sobolev Spaces}

Let $\Omega\subseteq\R^d$ be open and $H\coloneqq L_2(\Omega,\C^n)$.
In the following, we use the scalar product (and respective dual pairing)
\[\Scalr{u,v}\coloneqq\int_\Omega u(x)\cdot\overline{v(x)}\dx\text.\]
In case $d\ge3$, we put $p_*\coloneqq\frac{2d}{d-2}$; in case $d\le2$,
we fix an arbitrary $p_*\in(2,\infty)$. Let $V\subseteq W^{1,2}(\Omega,\C^n)$
be a closed subspace which is dense in $H$. We assume that $\Omega$ is such
that Sobolev's embedding theorem is valid in the sense that there is a
continuous embedding $V\subseteq L_{p_*}(\Omega,\C^n)$.

\begin{remark}
For the case that $\Omega$ is such that the dense embedding
$V\subseteq L_{p_*}(\Omega,\C^n)$ holds only with some smaller power
$p_*\in(2,\infty)$, all subsequent considerations hold as well
with this choice of $p_*$.
\end{remark}

\begin{lemma}\label{l:Lq}
Let $A$ be a Kato operator.
\begin{enumerate}
\item\label{i:Lqless}
Let $\gamma\in[0,1/2]$.
\begin{enumerate}
\item\label{ii:Lqbriless}
We have a continuous embedding
\(L_{q_\gamma}(\Omega,\C^n)\subseteq
H_\gamma'\cong\Hcb{\gamma}\cong H_{-\gamma}\)
with
\begin{equation}\label{e:qdef}
q_\gamma\coloneqq\Bigl(\frac12+\gamma-\frac{2\gamma}{p_*}\Bigr)^{-1}\quad
\text{$\Bigl({}=\frac{2d}{d+4\gamma}\in
\Bigl[\frac{2d}{d+2},2\Bigr]$ if $d\ge3\Bigr)$.}
\end{equation}
\item\label{ii:Lqgammaless}
If we have a continuous embedding $D(A)\subseteq L_p(\Omega,\C^n)$
$(1\le p<\infty)$, then we also have a continuous embedding
$H_{1-\gamma}\subseteq L_{p_\gamma}(\Omega,\C)$ with
\begin{equation}\label{e:pgamma}
p_\gamma\coloneqq\Bigl(\frac{2\gamma}{p_*}+\frac{1-2\gamma}p\Bigr)^{-1}\text.
\end{equation}
\end{enumerate}
\item\label{i:Lqlarg}
Let $\gamma\in[1/2,1]$.
\begin{enumerate}
\item\label{ii:Lqgammalarg}
We have a continuous embedding
$H_{1-\gamma}\subseteq L_{p_\gamma}(\Omega,\C)$ with
\begin{equation}\label{e:pgammas}
p_\gamma\coloneqq\Bigl(\gamma-\frac12+\frac{2-2\gamma}{p_*}\Bigr)^{-1}
\quad\text{$\Bigl({}=\frac{2d}{4\gamma+d-2}\in
\Bigl[\frac{2d}{d+2},2\Bigr]$ if $d\ge3\Bigr)$.}
\end{equation}
\item\label{ii:Lqbrilarg}
If we have a continuous embedding $D(A^*)\subseteq L_p(\Omega,\C^n)$
$(1\le p<\infty)$, then we also have a continuous embedding
$L_{q_\gamma}(\Omega,\C^n)\subseteq\Hcb{\gamma}\cong H_{-\gamma}$ with
\begin{equation}\label{e:qdefs}
q_\gamma\coloneqq\Bigl(1-\frac{2\gamma-1}p-\frac{2-2\gamma}{p_*}\Bigr)^{-1}
\text.
\end{equation}
\end{enumerate}
\end{enumerate}
\end{lemma}
\begin{proof}
By hypothesis, we have a continuous dense embedding
$V\subseteq L_{p_*}(\Omega,\C^n)$.
Hence, with $\frac1{p_*'}+\frac1{p_*}=1$ also the (Banach space) adjoint
embedding $L_{p_*'}(\Omega,\C^n)\subseteq V'$ is continous and dense.
In case $\gamma=1/2$ we have $q_\gamma=p_\gamma=p_*'$, and thus
the assertion~\iref{i:Lqlarg} follows.
In case $\gamma=0$ we have $q_\gamma=2$ and $p_\gamma=p$, and the
assertion~\iref{i:Lqless} is trivial. In case $\gamma\in(0,1/2)$,
we use~\cite[Theorem~1.18.4]{Triebel}, the fact that
$[\parameter,\parameter]_\theta$ is an interpolation functor
of order $\theta$ (see e.g.~\cite[Theorem~1.9.3(a)]{Triebel}),
and~\eqref{e:Hgammabri}. Then we have a continuous embedding
\[L_{q_\gamma}(\Omega,\C^n)\cong
[L_{p_*'}(\Omega,\C^n),L_2(\Omega,\C^n)]_{1-2\gamma}
\subseteq[V',H]_{1-2\gamma}\cong H_\gamma'\text,\]
which proves~\iiref{ii:Lqbriless}.

Defining $p'$ by $\frac1{p'}+\frac1p=1$, we find in view of
$\Hcc{1}=D(A^*)\subseteq L_p(\Omega,\C^n)$ that
$L_p'(\Omega,\C^n)=L_p(\Omega,\C^n)'\subseteq\Hcb{1}$.
This shows~\iref{i:Lqlarg} for $\gamma=1$, since $q_\gamma=p'$ and
$p_\gamma=2$. Moreover, for $\gamma\in(1/2,1)$, we find similarly as above
with~\eqref{e:Hinthalf} the continuous embedding
\[L_{q_\gamma}(\Omega,\C^n)=[L_{p'}(\Omega),L_{p_*'}(\Omega,\C^n)]_{2-2\gamma}
\subseteq[D(A^*)',V']_{2-2\gamma}\cong\Hcb{\gamma}\text,\]
which implies~\iiref{ii:Lqbrilarg}.
A similar argument shows with Proposition~\ref{p:powers} that in case
$\gamma\in(0,1/2)$
\[H_{1-\gamma}\cong[H_{1/2},H_1]_{1-2\gamma}\subseteq
[L_{p_*}(\Omega,\C^n),L_p(\Omega,\C^n)]_{1-2\gamma}\cong
L_{p_\gamma}(\Omega,\C^n)\text,\]
proving~\iiref{ii:Lqgammaless},
while in case $\gamma\in(1/2,1)$
\[H_{1-\gamma}\cong[H_0,H_{1/2}]_{2-2\gamma}\subseteq
[L_2(\Omega,\C^n),L_{p_*}(\Omega,\C^n)]_{2-2\gamma}\cong
L_{p_\gamma}(\Omega,\C^n)\text,\]
proving~\iiref{ii:Lqgammalarg} (all embeddings being countinuous).
\end{proof}

We assume also that the nonlinearity $f(t,\parameter)$
is given by a superposition operator induced by a function
$\widetilde f\colon[0,\infty)\times\Omega\times\C^n\to2^{\C^n}$, that is,
for each $t\in[0,\infty)$
\begin{equation}\label{e:superpos}
f(t,u)\coloneqq\Curly{v\colon\Omega\to\C^n\mid\text{$v$ measurable and
$v(x)\in\widetilde f(t,x,u(x))$ for almost all $x\in\Omega$}}\text.
\end{equation}
For the stability result, without loss of generality,
we will consider only the case $u_0=0$ and assume that $\widetilde f$ is
uniformly linearizable at $u=0$ in the following sense.
There are $r\in(1,\infty]$, a measurable
$\widetilde B\colon\Omega\to\C^{n\times n}$, and
a function $\widetilde g\colon(0,\infty)\times\Omega\times\C^n\to2^{\C^n}$ with
\[\widetilde f(t,x,u)=\widetilde B(x)u+\widetilde g(t,x,u)\quad
\text{for all $(t,x,u)\in(0,\infty)\times\Omega\times\C^n$}\]
such that
\begin{equation}\label{e:g0limit}
\lim_{\Abs u\to0}\,
\frac{\sup\bigl\{\Abs v: v\in\widetilde g\bigl(
(0,\infty)\times\Curly x\times\Curly u\bigr)\bigr\}}{\Abs u}=0
\end{equation}
for almost all $x\in\Omega$. Moreover, we assume that there
is $C_0\in(0,\infty)$ such that
\begin{equation}\label{e:g0est}
\sup\bigl\{\Abs v:
v\in\widetilde g\bigl((0,\infty)\times\Curly x\times\Curly u\bigr)\bigr\}\le
C_0\cdot(\Abs u+\Abs u^{\sigma})\quad\text{for all $u\in\C^n$}
\end{equation}
for almost all $x\in\Omega$ and some $\sigma\in(1,\infty)$. We define a
corresponding multiplication operator $B$ by
\begin{equation}\label{e:Bmultiplier}
Bu(x)\coloneqq\widetilde B(x)u(x)\quad\text{for all $x\in\Omega$.}
\end{equation}
With this notation, the following holds.

\begin{proposition}\label{p:diffop}
Let $A$ be a Kato operator and $u_0=0$. Suppose that
\begin{equation}\label{e:mesOmega}
\text{$\Omega$ has finite (Lebesgue) measure}
\end{equation}
and that $r\in[1,\infty]$ and $\sigma\in(0,\infty)$ are such that
$\widetilde B\in L_r(\Omega,\C^{n\times n})$ and~\eqref{e:g0limit}
and~\eqref{e:g0est} hold.
\begin{enumerate}
\item\label{i:L2}
Let $\alpha=0$. Assume
\begin{equation}\label{e:rboring}
\begin{cases}
r=2&\text{if $d=1$},\\
r>2&\text{if $d=2$,}\\
\displaystyle
r=\frac{2p_*}{p_*-2}\quad({}=d)&\text{if $d\ge3$,}
\end{cases}
\end{equation}
and
\begin{equation}\label{e:pL2}
\begin{cases}
\sigma=2&\text{if $d=1$,}\\
\sigma<2&\text{if $d=2$,}\\
\displaystyle
\sigma=2-\frac2{p_*}\quad\Bigl({}=1+\frac2d\Bigr)&\text{if $d\ge3$.}
\end{cases}
\end{equation}
Then for every $\gamma\in[1/2,1)$ there holds
$f\colon[0,\infty)\times H_\alpha\to2^{\Hcb{\gamma}}$, and
the hypothesis~\Hypo{B_\gamma} of Theorems~\ref{t:stability1}
and~\ref{t:stability2} is satisfied with $H_\alpha=L_2(\Omega,\C^n)$.
\item\label{i:L2a}
Let $\alpha=0$. Assume that the embedding
$D(A^*)\subseteq L_p(\Omega,\C^n)$ is continuous for some $p\in(p_*,\infty)$,
and
\begin{equation}\label{e:r2lower}
r>\frac{2p}{p-2}\quad\text{and}\quad\sigma<2-\frac2p\text.
\end{equation}
Then
\begin{equation}\label{e:gamma0L2}
\gamma_0\coloneqq\frac{(\sigma-2)p_*p-2p_*+4p}{4(p-p_*)}<1\text,\quad
\gamma_1\coloneqq\frac{2pp_*-r(pp_*+2p_*-4p)}{4(p-p_*)r}<1\text,
\end{equation}
and for every $\gamma\in[\max\Curly{\gamma_0,\gamma_1,1/2},1)$
the same conclusion as in~\iref{i:L2} is valid.
\item\label{i:W1}
Let $\alpha=1/2$. Suppose that $\widetilde B\in L_r(\Omega,\C^{n\times n})$
with some
\begin{equation}\label{e:rstrict}
r>\frac{p_*}{p_*-2}\quad\text{$\Bigl({}=\frac d2$ if $d\ge3\Bigr)$,}
\end{equation}
and that~\eqref{e:g0limit} and~\eqref{e:g0est} hold with some
\begin{equation}\label{e:pW2}
\sigma<p_*-1\quad\text{$\Bigl({}=\frac{d+2}{d-2}$ if $d\ge3\Bigr)$.}
\end{equation}
Then
\begin{equation}\label{e:gamma0}
\gamma_0\coloneqq\frac{2\sigma-p_*}{2p_*-4}<\frac12\quad\text{and}\quad
\gamma_1\coloneqq\frac{p_*}{r(p_*-2)}-\frac12<\frac12\text,
\end{equation}
and for every $\gamma\in[\max\Curly{0,\gamma_0,\gamma_1},1/2)$ we have
$f\colon[0,\infty)\times H_\alpha\to2^{H_\gamma'}$, and
the hypothesis~\Hypo{B_\gamma} of
Theorems~\ref{t:stability1} and~\ref{t:stability2} is satisfied with
$H_\alpha=V$.
\end{enumerate}
\end{proposition}
\begin{proof}
In case~\iref{i:L2}, it is no loss of generality to assume $\gamma=1/2$,
and we assume first $d\ge3$. In cases~\iref{i:L2} and~\iref{i:L2a},
we put $\widetilde p=2$ and define $q_\gamma$ by~\eqref{e:qdefs},
while in case~\iref{i:W1}, we put $\widetilde p=p_*$ and
define $q_\gamma$ by~\eqref{e:qdef}. Then we put
$U\coloneqq L_{\widetilde p}(\Omega,\C^n)$ and
$V_\gamma\coloneqq L_{q_\gamma}(\Omega,\C^n)$.
Letting $r$ satisfy~\eqref{e:rboring},~\eqref{e:r2lower},
or~\eqref{e:rstrict}, and requiring $\gamma\ge\gamma_1$ with $\gamma_1$ as
in~\eqref{e:gamma0L2} or~\eqref{e:gamma0} in the respective cases, we find
\[\frac1{q_\gamma}\ge\frac1{\widetilde p}+\frac1r\text,\]
and so we obtain from the (generalized) H\"{o}lder inequality that
$B\colon U\to V_\gamma$ is bounded. Since we have a bounded embedding
$H_\alpha\subseteq U$, we obtain from Lemma~\ref{l:Lq} that
$B\colon H_\alpha\to\Hcb{\gamma}$ is bounded.

Moreover, letting $\sigma$ satisfy~\eqref{e:pL2},~\eqref{e:r2lower},
or~\eqref{e:pW2}, and requiring $\gamma\ge\gamma_0$ with $\gamma_0$ as
in~\eqref{e:gamma0L2} or~\eqref{e:gamma0} in the respective cases, we find
$\sigma\le\widetilde p/q_\gamma$.
Hence, the superposition operator $g$ generated by $\widetilde g$
satisfies $g\colon[0,\infty)\times U\to2^{V_\gamma}$ and
\[\lim_{\Norm u_U\to0}\,\frac{\sup\bigl\{\Norm v_{V_\gamma}:
v\in g\bigl((0,\infty)\times\Curly u\bigr)\bigr\}}{\Norm u_U}=0\text,\]
see~\cite[Theorem~4.14]{VaethSupAtomI}. Since we have continuous embeddings
$H_\alpha\subseteq U$ and $V_\gamma\subseteq\Hcb{\gamma}$ (Lemma~\ref{l:Lq}),
the condition~\Hypo{B_\gamma} is proved.

Case~\iref{i:L2} with $d=2$ is treated in a similar way (with a sufficiently
large $p_*$), and for $d=1$ we can put $q_\gamma=1$ in the above calculation,
since in this case we have still a continuous embedding
$V_\gamma\subseteq\Hcb{\gamma}$ by the continuity of the embedding
$\Hcb{\gamma}\subseteq H_{1/2}\subseteq L_\infty(\Omega,\C^n)$.
\end{proof}

\begin{remark}\label{r:Cembed}
The last observation in the proof extends to a more general situation:
If $\gamma\in[0,1/2]$ is such that the embedding
$H_\gamma\subseteq L_\infty(\Omega,\C^n)$ is continuous,
then the conclusion of Proposition~\ref{p:diffop}\iref{i:L2}
is valid with $r=\sigma=2$ (we put $q_\gamma=1$ in the proof).
\end{remark}

\begin{remark}
For $d\ge3$ assertion~\iref{i:L2a} of Proposition~\ref{p:diffop} requires
strictly less about $r$ and $\sigma$ than assertion~\iref{i:L2}, because
in view of $p>p_*$ there holds
\[\frac{2p}{p-2}<\frac{2p_*}{p_*-2}\quad\text{and}\quad
2-\frac2p>2-\frac2{p^*}\text.\]
\end{remark}

\begin{remark}
In case $d\ge3$ the quantities $\gamma_0$ and $\gamma_1$ in~\eqref{e:gamma0}
have the form
\begin{equation}\label{e:gamma1}
\gamma_0=\frac{(d-2)\sigma-d}{4}\quad\text{and}\quad
\gamma_1=\frac12\Bigl(\frac dr-1\Bigr)\text.
\end{equation}
\end{remark}

\begin{remark}
Proposition~\ref{p:diffop}\iref{i:W1} holds also with $\gamma=1/2$.
Moreover, for $\gamma=1/2$ one does not have to require that
the inequalities in~\eqref{e:rstrict} or~\eqref{e:pW2} are strict. However,
the choice $\gamma=1/2$ violates the hypothesis~\eqref{e:gammaalpha} of
Theorems~\ref{t:stability1} and~\ref{t:stability2} if $\alpha=1/2$.
\end{remark}

\begin{remark}
Hypothesis~\eqref{e:mesOmega} is obviously needed for the
assertion~\iref{i:W1} of Proposition~\ref{p:diffop}. However, we used this
hypothesis also for the assertion~\iref{i:L2} when we
applied~\cite[Theorem~4.14]{VaethSupAtomI}.
If hypothesis~\eqref{e:mesOmega} fails, one can apply other criteria
for the differentiability of superposition operators like
e.g.~\cite[Theorem~4.9]{VaethSupAtomI}, but we do not formulate
corresponding results here.
\end{remark}

While Proposition~\ref{p:diffop} gives a sufficient condition for
the hypothesis~\Hypo{B_\gamma}, this is not sufficient to apply
Theorem~\ref{t:stability2}. For the latter, one also has to estimate
all $\gamma$-weak eigenvalues of $A-B$, and the latter in turn is
usually simpler if one knows that all $\gamma$-weak eigenvalues of $A-B$
are eigenvalues of $A-B$. For the operator $B$ from~\eqref{e:Bmultiplier},
this is the content of the following result.

\begin{proposition}\label{p:weakev}
Suppose~\eqref{e:mesOmega}.
Let $B$ have the form~\eqref{e:Bmultiplier} with some
$\widetilde B\in L_r(\Omega,\C^n)$, $r\in[1,\infty]$.
\begin{enumerate}
\item\label{i:rgood}
If $r$ satisfies~\eqref{e:rboring}, then $B|_V\colon V\to H$ is bounded.
\item\label{i:rsemi}
If $A$ is a Kato operator, $\gamma\in[1/2,1)$, and
\begin{equation}\label{e:rsemi}
\gamma\le\widetilde\gamma_0\coloneqq
\begin{cases}\displaystyle1-\frac{p_*}{(p_*-2)r}&\text{if $r<\infty$,}\\
1&\text{if $r=\infty$,}\end{cases}
\end{equation}
then $B|_{H_{1-\gamma}}\colon H_{1-\gamma}\to H$.
\item\label{i:rbad}
If $A$ is a Kato operator,
\begin{equation}\label{e:rinterest}
2<r<\frac{2p_*}{p_*-2}\quad\text{$\bigl({}=d$ if $d\ge3\bigr)$,}
\end{equation}
and if there is $p\ge\frac{2r}{r-2}$ $({}>p_*)$ with
$D(A)\subseteq L_p(\Omega,\C^n)$, then
\begin{equation}\label{e:rgamma}
\widetilde\gamma_p\coloneqq\frac12\Bigl(\frac12-\frac1r-\frac1p\Bigr)\cdot
\Bigl(\frac1{p_*}-\frac1p\Bigr)^{-1}\in[0,1/2)\text,
\end{equation}
and for all $\gamma\le\widetilde\gamma_p$ the operator
$B|_{H_{1-\gamma}}\colon H_{1-\gamma}\to H$ is bounded.
\end{enumerate}
If $A$ is a Kato operator, the hypotheses of
either~\iref{i:rgood},~\iref{i:rsemi}, or~\iref{i:rbad}
are satisfied and $\gamma\le1/2$, $\gamma\le\widetilde\gamma_0$,
or $\gamma\le\widetilde\gamma_p$, respectively,
then $\lambda$ is a $\gamma$-weak eigenvalue of $A-B$
if and only if $\lambda$ is an eigenvalue of $A-B$.
\end{proposition}
\begin{proof}
In case~\iref{i:rgood} with $d\ge 3$, we apply in view of
\[\frac12=\frac1{p_*}+\frac1r\]
the (generalized) H\"{o}lder inequality to obtain that
$B\colon L_{p_*}(\Omega,\C^n)\to L_2(\Omega,\C^n)$ is bounded and thus
$B\colon V\to H$ is bounded. Case~\iref{i:rgood} with $d\ge 2$ is
similar (with sufficiently large $p_*$), and for $d=1$ one can formally
put $p_*=\infty$ by the continuity of the embedding
$H_{1/2}\subset C(\overline\Omega,\C^n)$.

In case~\iref{i:rsemi} and~\iref{i:rbad}, we define $p_\gamma$
by~\eqref{e:pgammas} or~\eqref{e:pgamma}, respectively,
and observe that, due to~\eqref{e:rsemi} or~\eqref{e:rgamma}, respectively,
we have the estimate
\[\frac12\ge\frac1{p_\gamma}+\frac1r\text.\]
Hence, by the (generalized) H\"{o}lder inequality,
$B\colon L_{p_\gamma}(\Omega,\C^n)\to L_2(\Omega,\C^n)$ is bounded,
and thus also $B\colon H_{1-\gamma}\to H$ is bounded by Lemma~\ref{l:Lq}.
The last assertion follows from Proposition~\ref{p:BVH} and
Remark~\ref{r:BVH}.
\end{proof}

If one is interested in stability in $H$ (the case $\alpha=0$),
one should consider Proposition~\ref{p:diffop} part~\iref{i:L2}
or~\iref{i:L2a}. In the former case, Proposition~\ref{p:weakev}\iref{i:rgood}
is automatically satisfied, and in the latter case one would like to
apply Proposition~\ref{p:weakev}\iref{i:rsemi}. In the latter case,
$\gamma\in[1/2,1)$ has to satisfy
$\gamma_i\le\gamma\le\widetilde\gamma_0$ for $i=0,1$
with $\gamma_i$ from~\eqref{e:gamma0L2}. Obviously, $\gamma_1$ and
$\widetilde\gamma_0$ depend monotonically on $r$, and
$\gamma_1<\widetilde\gamma_0$ if $r$ is sufficiently large, and then
$\gamma_0<\widetilde\gamma_0$ if $\sigma$ is sufficiently small,
so that Proposition~\ref{p:diffop}\iref{i:L2a} and
Proposition~\ref{p:weakev}\iref{i:rsemi} apply simultaneously
for all $\gamma$ from some proper interval (if $r$ is sufficiently large).

If one is interested in stability in $V$ (the case $\alpha=1/2$),
one should consider Proposition~\ref{p:diffop}\iref{i:W1}.
In this case, the hypothesis of Proposition~\ref{p:weakev}\iref{i:rgood}
means an additional requirement for $r$. The purpose of
Proposition~\ref{p:weakev}\iref{i:rbad} is to relax this requirement.
However, it is not immediately clear whether this relaxed requirement
applies in the situation of  Proposition~\ref{p:diffop}\iref{i:W1},
since then $\gamma\in[0,1/2]$ needs to satisfy
$\gamma_i\le\gamma\le\widetilde\gamma_p$ for $i=0,1$
with $\gamma_i$ from~\eqref{e:gamma0}. Although $\gamma_1$ and
$\widetilde\gamma_p$ depend monotonically on $r$ and satisfy
$\gamma_1<\widetilde\gamma_p$ if $r$ is sufficiently large,
one cannot choose $r$ arbitrarily large in view of~\eqref{e:rinterest}:
Otherwise already the additional requirement of
Proposition~\ref{p:weakev}\iref{i:rgood} is satisfied.
In fact, the following observation may be somewhat discouraging at
a first glance.

\begin{remark}
If~\eqref{e:rinterest} holds, then the term $\widetilde\gamma_p$
in~\eqref{e:rgamma} is strictly increasing with respect to
$p\ge\frac{2r}{r-2}$. In particular,
\[\widetilde\gamma_\infty\coloneqq
\sup_{p\in[\frac{2r}{r-2},\infty)}\,\widetilde\gamma_p=
\lim_{p\to\infty}\,\widetilde\gamma_p=\frac{r-2}{4r}p_*\text.\]
Thus, even if we know that $D(A)\subseteq L_p(\Omega,\C^n)$ for
every $p\in(1,\infty)$, we still have
$\gamma<\widetilde\gamma_\infty$, and the latter can be
arbitrarily small if $r$ is sufficiently close to~$2$.
\end{remark}

Nevertheless we will show in the following remark that
Proposition~\ref{p:diffop}\iref{i:W1} and
Proposition~\ref{p:weakev}\iref{i:rbad} apply simultaneously with
the same $\gamma$ provided that $r$ is not ``too'' small and
$\sigma$ is not ``too'' large.

\begin{remark}
Suppose that Sobolev's embedding theorem holds in the sense
described earlier and, moreover, that we have a continuous
embedding $D(A)\subseteq L_p(\Omega,\C^n)$ with $p=\frac{2d}{d-4}$
in case $d\ge5$ and any $p\in(p_*,\infty)$ in case $d\le4$. For
instance, by standard Sobolev embedding theorems
(see~\cite[Theorem~1.4.5]{MaziaSob}), this is the case if
$D(A)\subseteq W^{2,2}(\Omega,\C^n)$.
Proposition~\ref{p:weakev}\iref{i:rbad} applies with
\[\begin{cases}
\text{$r\in[\frac d2,d]$,
$\gamma\le\widetilde\gamma_{2d/(d-4)}=1-\frac d{2r}$}&\text{if $d\ge5$,}\\
\text{$r\in(2,d)$, $\gamma<\widetilde\gamma_\infty=\frac{r-2}{4r}p_*$}&
\text{if $d=3,4$.}
\end{cases}\]
In view of~\eqref{e:gamma1} it follows that if
\[\begin{cases}
\text{$r\in[\frac23d,d)$ and $\sigma\le\frac{(d+4)r-2d}{(d-2)r}$}&
\text{if $d\ge5$,}\\
\text{$r\in(\frac{d^2}{2d-2},d)$ and $\sigma<\frac{d^2r-4d}{(d-2)^2r}$}&
\text{if $d=3,4$,}
\end{cases}\]
then Proposition~\ref{p:diffop}\iref{i:W1} applies with
\[\begin{cases}\max\Curly{\gamma_0,\gamma_1}\le\widetilde\gamma_{2d/(d-4)}&
\text{if $d\ge5$,}\\
\max\Curly{\gamma_0,\gamma_1}<\widetilde\gamma_\infty&
\text{if $d=3,4$.}
\end{cases}\]
Hence, in these cases there exists $\gamma\in[0,1/2)$ for which
Proposition~\ref{p:diffop}\iref{i:W1} and
Proposition~\ref{p:weakev}\iref{i:rbad} apply simultaneously.
\end{remark}

A result similar to Proposition~\ref{p:diffop} holds for a Lipschitz condition.
We assume that $\widetilde f\colon[0,\infty)\times\Omega\times\C^n\to\C^n$
is single-valued. Let $\widetilde p\ge1$, $\sigma>0$, and $\gamma\in[0,1/2]$.
We define $q_\gamma$ by~\eqref{e:qdef}.
We assume that for each $t_0\in[0,\infty)$ there are
$L_{t_0}\ge0$, $\sigma_{t_0}>0$, and a neighborhood
$I\subseteq[0,\infty)$ of $t_0$
such that for each $t\in I$ there are measurable
$a_t,b_t\colon\Omega\to[0,\infty)$ with
\[\text{$\int_\Omega a_t(x)^{\widetilde p}\dx\le1$ and
$\int_\Omega b_t(x)^{q_\gamma}\dx\le1$}\]
such that for almost all $x\in\Omega$ the uniform
(for all $u,v\in\C^n$) estimate
\begin{equation}\label{e:tildefL}
\Abs{\widetilde f(t,x,u)-\widetilde f(t,x,v)}\le
L_{t_0}\cdot\bigl(a_t(x)+\Abs u+\Abs v\bigr)^{\sigma-1}\Abs{u-v}
\end{equation}
holds and such that for each $t,s\in I$ we have
for almost all $x\in\Omega$ the uniform (for all $u\in\C^n$) estimate
\begin{equation}\label{e:tildefH}
\Abs{\widetilde f(t,x,u)-\widetilde f(s,x,u)}\le
L_{t_0}\bigl(b_t(x)+b_s(x)+\Abs u^\sigma\bigr)\Abs{t-s}^{\sigma_{t_0}}\text.
\end{equation}
Finally, we assume that
\begin{equation}\label{e:f0good}
\text{$\widetilde f(t,\parameter,u)$ is measurable for all
$(t,u)\in[0,\infty)\times\C^n$, and
$\widetilde f(0,\parameter,0)\in L_{q_\gamma}(\Omega,\C^n)$.}
\end{equation}

\begin{proposition}\label{p:Lipop}
Let $A$ be a Kato operator, and assume~\eqref{e:mesOmega}.
Assume one of the following:
\begin{enumerate}
\item\label{i:LL2}
Let $\alpha=0$ and $\gamma\in[1/2,1)$. Suppose that
conditions~\eqref{e:tildefL},~\eqref{e:tildefH}, and~\eqref{e:f0good}
hold with $\widetilde p=2$ and with $\sigma$ from~\eqref{e:pL2}.
\item
Let $\alpha=0$, and assume that the embedding
$D(A^*)\subseteq L_p(\Omega,\C^n)$ is continuous for some $p\in(p_*,\infty)$,
Let $\sigma$ satisfy~\eqref{e:r2lower}, and thus $\gamma_0$
from~\eqref{e:gamma0L2} satisfies $\gamma_0<1$. Let
$\gamma\in[\max\Curly{\gamma_0,0},1)$,
and suppose that conditions~\eqref{e:tildefL},~\eqref{e:tildefH},
and~\eqref{e:f0good} hold with $\widetilde p=2$.
\item\label{i:LW1}
Let $\alpha=1/2$. Let $\sigma$
satisfy~\eqref{e:pW2}, and thus $\gamma_0$ from~\eqref{e:gamma0}
satisfies $\gamma_0<1/2$. Let $\gamma\in[\max\Curly{\gamma_0,0},1/2)$,
and suppose that conditions~\eqref{e:tildefL},~\eqref{e:tildefH},
and~\eqref{e:f0good} hold with $\widetilde p=p_*$.
\end{enumerate}
Then $f$ maps $[0,\infty)\times H_\alpha$ into $H_{-\gamma}$ and satisfies a
right local H\"{o}lder-Lipschitz condition~\eqref{e:goodLip}
and is left-locally bounded into $H_{-\gamma}$.
\end{proposition}
\begin{proof}
We use the notation of the proof of Proposition~\ref{p:diffop}.
Note that~\eqref{e:f0good} implies in view of~\eqref{e:tildefH} by a
straightforward estimate that $f(t,0)\in V_\gamma$ for every $t>0$.
From~\cite[Appendix]{VaethQuittner} we obtain together with~\eqref{e:tildefL}
that for each $t\in I$ the function $f(t,\parameter)$ maps $U$ into $V_\gamma$
and satisfies a Lipschitz condition on every bounded set $M\subseteq U$
with Lipschitz constant being independent of $t\in I$. Using~\eqref{e:tildefH},
we find by a straighforward estimate that
\[\Norm{f(t,u)-f(s,u)}_{V_\gamma}\le C_{M,t_0}
\Abs{t-s}^{\sigma_{t_0}}\quad\text{for all $t,s\in I$, $u\in M$,}\]
where $C_{M,t_0}$ is independent of $t,s\in I$ and $u\in M$.
Combining both assertions and the triangle inequality, we obtain that
$f\colon[0,\infty)\times U\to V_\gamma$ satisfies a right H\"{o}lder-Lipschitz
condition and is left-locally bounded into $V_\gamma$.
Since we have bounded embeddings $H_\alpha\subseteq U$ and
$V_\gamma\subseteq\Hcb{\gamma}\cong H_{-\gamma}$ by~\eqref{e:dual},
the assertion follows.
\end{proof}

\begin{remark}\label{r:CembedLip}
If $\alpha=0$ and $\gamma\in[0,1)$ is such that the embedding
$H_\gamma\subseteq L_\infty(\Omega,\C^n)$ is continuous,
then the conclusion of Proposition~\ref{p:Lipop} is also valid
(with the same proof and $q_\gamma=1$, cf.\ Remark~\ref{r:Cembed}).
\end{remark}

\subsection{Semilinear Parabolic PDEs}

Let $\Omega\subseteq\R^d$ be a bounded domain with Lipschitz boundary
$\partial\Omega$. Let $\Gamma_D,\Gamma_N\subseteq\partial\Omega$ be disjoint
and measurable (with respect to the $(d-1)$-dimensional Hausdorff measure)
and such that
\begin{equation}\label{e:union}
\text{$(\partial\Omega)\setminus(\Gamma_D\cup\Gamma_N)$ is a null set.}
\end{equation}
It is explicitly admissible that $\Gamma_D=\emptyset$ or $\Gamma_N=\emptyset$.
Given $a_{j,k},b_j\in L_\infty(\Omega,\C^{n\times n})$ $(j,k=1,\dotsc,d)$ and
$\widetilde f\colon[0,\infty)\times\Omega\times\C^n\to2^{\C^n}$,
we consider the semilinear PDE
\begin{equation}\label{e:dgl}
\frac{\partial u}{\partial t}+Pu=\widetilde f_0(t,x,u)\quad
\text{on $\Omega$,}
\end{equation}
where
\[Pw\coloneqq-\sum_{j,k=1}^d\frac{\partial}{\partial x_j}
\Bigl(a_{j,k}(x)\frac{\partial w(x)}{\partial x_k}\Bigr)+
\sum_{j=1}^db_j(x)\frac{\partial w(x)}{\partial x_j}\text.\]
We impose the mixed boundary condition
\begin{equation}\label{e:bdry}
\begin{cases}
u=0&\text{on $\Gamma_D$,}\\
\displaystyle\sum_{j,k=1}^d\nu_ja_{j,k}\frac{\partial u}{\partial x_k}=0&
\text{on $\Gamma_N$,}
\end{cases}
\end{equation}
where $\nu(x)=(\nu_1(x),\dotsc,\nu_n(x))$ denotes the outer normal at
$x\in\partial\Omega$.

We put $H\coloneqq L_2(\Omega,\C^n)$ and
\[V\coloneqq\Curly{u\in W^{1,2}(\Omega,\C^n):
\text{$u|_{\Gamma_D}=0$ in the sense of traces}}\text,\]
equipping $V$ with the norm of $W^{1,2}(\Omega,\C^n)$.

Our main assumption is as follows.
\begin{description}
\item[\Hypo{C}] G\aa rding's inequality holds, that is,
there are $c,\widetilde c>0$ with
\begin{equation}\label{e:garding}
\Real\sum_{j,k=1}^d\int_\Omega
\Bigl(a_{j,k}(x)\frac{\partial u(x)}{\partial x_k}\Bigr)
\cdot\overline{\frac{\partial u(x)}{\partial x_j}}\dx\ge
c\Norm{\nabla u}_{L_2(\Omega,\C^{dn})}^2-
\widetilde c\Norm u_{L_2(\Omega,\C^n)}^2
\end{equation}
for all $u\in V$. Moreover, at least one of the following holds:
\begin{enumerate}
\item\label{i:selfadjoint}
$a_{j,k}(x)=(a_{k,j}(x))^*$ for almost all $x\in\Omega$ and all
$j,k=1,\dotsc,d$.
\item\label{i:egert}
G\aa rding's inequality~\eqref{e:garding} holds even with $\widetilde c=0$.
Moreover, $\Gamma_D$ satisfies the geometric hypotheses described
in~\cite[Assumption~9.1]{EgertKatoMixed}.
\item\label{i:agranovich}
The matrices $\Real\Bigl(\sum\limits_{j,k=1}^da_{j,k}(x)\xi_j\xi_k\Bigr)$ are
positive definite, uniformly with respect to all $x\in\Omega$ and
$\xi\in\mathbb R^d$ with $\Abs\xi=1$,
$a_{j,k}\in C^1(\overline\Omega,\C^{n\times n})$,
$b_j\in\Lip(\overline\Omega,\C^{n\times n})$ for all
$j,k=1,\dotsc,d$, $\Gamma_D$ and $\Gamma_N$ are open in $\partial\Omega$
domains, and the set~\eqref{e:union} is a Lipschitz manifold of
dimension $d-2$.
\end{enumerate}
\end{description}

For a discussion of various algebraic conditions that are
sufficient for G\aa rding's inequality~\eqref{e:garding},
we refer the reader to e.g.~\cite{McLean,AgranovichGarding}.

By a standard estimate, we obtain from G\aa rding's
inequality~\eqref{e:garding} that the form
\[a(u,v)\coloneqq\int_\Omega
\left(\sum_{j,k=1}^d\Bigl(a_{j,k}(x)\frac{\partial u(x)}{\partial x_k}\Bigr)
\cdot\overline{\frac{\partial v(x)}{\partial x_j}}+
\sum_{j=1}^d\Bigl(b_j(x)\frac{\partial u(x)}{\partial x_j}+Mu(x)\Bigr)\cdot
\overline{v(x)}\right)\dx\]
satisfies~\eqref{e:strongacc} if $M\ge0$ is sufficiently large.
Keeping such an $M$ fixed, we now introduce the function
\[\widetilde f(t,x,u)\coloneqq\widetilde f_0(t,x,u)+Mu\]
and define strong (weak, mild) solutions of~\eqref{e:dgl},~\eqref{e:bdry}
as strong (weak, mild) solutions of~\eqref{e:semilinear} with the
superposition operator~\eqref{e:superpos}. A connection between solutions
of~\eqref{e:dgl},~\eqref{e:bdry} and~\eqref{e:semilinear} is
described in e.g.~\cite[Theorem~4.4.4]{VaethIdeal}.

\begin{theorem}\label{t:dgl}
Assume that hypothesis~\Hypo{C} holds. Then the operator $A$
associated with $a$ is a Kato operator.

Moreover, let $(\widetilde f,\alpha,\gamma)$ satisfy the
hypotheses of Proposition~\ref{p:diffop} part~\iref{i:L2} or~\iref{i:L2a}
(or~\iref{i:W1}), and suppose that there is some $\lambda_0>0$
such that every $\gamma$-weak eigenvalue $\lambda$ of $A-B$
satisfies $\Real\lambda\ge\lambda_0$. Then $u_0=0$ is asymptotically stable
in $H_\alpha=L_2(\Omega,\C^n)$ (or $H_\alpha=W^{1,2}(\Omega,\C^n)$)
in the sense that for every $\varepsilon>0$ there is
$\delta>0$ such that any $\gamma$-mild solution $u\in C([0,\infty),H_\alpha)$
of~\eqref{e:dgl},~\eqref{e:bdry} with
$\Norm{u(0,\parameter)}_{H_\alpha}\le\delta$ satisfies
$\Norm{u(t,\parameter)}_{H_\alpha}\le\varepsilon$ for all $t\ge0$,
and $\Norm{u(t,\parameter)}_{H_\alpha}\to0$ exponentially fast as
$t\to\infty$.

If in addition $\widetilde f$ satisfies the hypothesis of
Proposition~\ref{p:Lipop} part~\iref{i:LL2} (or~\iref{i:LW1}),
then for every $u_0\in H_\alpha$ there is a unique
$\gamma$-mild solution $u\in C([0,\infty),H_\alpha)$
of~\eqref{e:dgl},~\eqref{e:bdry} with $u(0,\cdot)=u_0$.
\end{theorem}

\begin{remark}
We emphasize that under the additional assumptions mentioned
in Proposition~\ref{p:weakev}, it suffices to consider eigenvalues of
$A-B$ instead of $\gamma$-weak eigenvalues. Note that $A-B$ is actually
independent of $M$ (because the terms with $M$ cancel).
\end{remark}

\begin{proof}
Assume first that $b_1=\dotsm=b_d=0$. Then $A$ is a Kato operator. Indeed,
in case~\Hypo{C}\iref{i:selfadjoint}, this follows from
Proposition~\ref{p:Aassoc} or by Theorem~\ref{t:quasiKato},
because $a$ is symmetric. In case~\Hypo{C}\iref{i:egert}, this follows
from the main result of~\cite{EgertKatoMixed}, and in
case~\Hypo{C}\iref{i:agranovich} this follows from the main result
of~\cite{Agranovich} in view of~\cite{AgranovichMixed}.

Since neither the space $H_1$ nor its topology depends on $M$ or $b_j$,
we obtain from Proposition~\ref{p:powers} that also the space
$H_{1/2}\cong[H,H_1]_{1/2}$ does not depend on $M$ or $b_j$, and so we obtain
from the special case $b_1=\dotsm=b_d=0$ also in the general case
that $A$ is a Kato operator.

Note that if the hypothesis of Proposition~\ref{p:diffop}\iref{i:L2} is
satisfied, then also the hypothesis of Proposition~\ref{p:weakev}\iref{i:rgood}
is satisfied. Hence, the assertion follows from Theorem~\ref{t:stability2}.
\end{proof}

\begin{remark}\label{r:Cembedt}
In Theorem~\ref{t:dgl}, the hypotheses of Proposition~\ref{p:diffop}
part~\iref{i:L2} or~\iref{i:L2a} can also be replaced by the hypothesis
of Remark~\ref{r:Cembed}.
\end{remark}

\begin{example}\label{x:hyperbola}
Let $\Omega\subseteq\R^d$ be bounded with a Lipschitz boundary,
$\Gamma_D,\Gamma_N\subseteq\partial\Omega$ be measurable with~\eqref{e:union}.
Let $f_1,f_2\colon\C^2\to\C$ be continuous with $f_i(0)=0$,
and suppose that there are $L\ge0$ and $\rho>0$ with
\begin{equation}\label{e:fuvLip}
\Abs{f_i(u)-f_i(v)}\le L\bigl(1+\Abs{u}+\Abs{v}\bigr)^\rho
\Abs{u-v}
\end{equation}
for all $u\in\C^2$. Assume that $(b_{i1},b_{i2})=f_i'(0)$ exist for $i=1,2$,
are real, and satisfy the sign conditions
\[b_{11}>0\text,\quad b_{11}+b_{22}<0\text,\quad b_{11}b_{22}-b_{12}b_{21}>0
\text.\]
For $d_1,d_2>0$, we consider the reaction-diffusion system
\begin{equation}\label{e:rctdiff}
\frac{\partial u_j}{\partial t}=d_j\Delta u_j+f_j(u_1,u_2)\quad
\text{on $\Omega$ for $j=1,2$,}
\end{equation}
with mixed boundary conditions (for $u=(u_1,u_2)$)
\begin{equation}\label{e:bdrys}
\text{$u=0$ on $\Gamma_D$,}\quad
\text{$\frac{\partial u}{\partial\nu}=0$ on $\Gamma_N$.}
\end{equation}
Let $\kappa_k>0$ $(k=1,2,\dotsc)$ denote the nonzero eigenvalues of
$\Delta$ with boundary conditions~\eqref{e:bdrys}; if $\Gamma_D$ is a null set,
do \emph{not} include the trivial eigenvalue $\kappa_0=0$ into this sequence.
Suppose $(d_1,d_2)$ lies to the right/under the envelope of the hyperbolas
\[C_k=\Curly{(d_1,d_2):\text{$d_1,d_2>0$ and
$(\kappa_kd_1-b_{11})(\kappa_kd_2-b_{22})=b_{12}b_{21}$}}\text,\]
that is, $(d_1,d_2)$ belongs to
\begin{equation}\label{e:DS}
\bigcap_{k=1}^\infty\Bigl\{(d_1,d_2):\text{$d_1\ge\kappa_k^{-1}b_{11}$ or
$d_2<\frac{\kappa_k^{-2}b_{12}b_{21}}{d_1-\kappa_k^{-1}b_{11}}+
\frac{b_{22}}{\kappa_k}$}\Bigr\}\text,
\end{equation}
see Figure~\ref{f:hyperbolas}.
\newcommand{\myhyperbola}[1]{-0.3 #1 dup mul mul x #1 sub div
 -0.05 #1 mul add}%
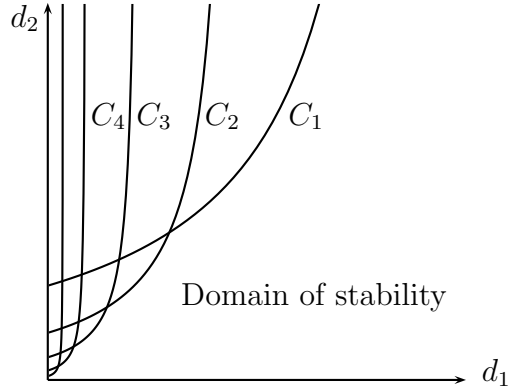
\begin{figure}[ht]
\centering
\begin{pspicture}(-1,0)(6,5.5)
\rput(5.9,.1){$d_1$}\rput(-.3,4.8){$d_2$}
\psline{->}(0,0)(0,5)\psline{->}(0,0)(5.5,0)
\begin{psclip}{\psframe[linestyle=none](0,0)(5.5,5)}
\psplot{0}{.199}{\myhyperbola{0.2}}
\psplot{0}{.49}{\myhyperbola{0.5}}\rput(.8,3.5){$C_4$}
\psplot{0}{1.199}{\myhyperbola{1.2}}\rput(1.4,3.5){$C_3$}
\psplot{0}{2.499}{\myhyperbola{2.5}}\rput(2.3,3.5){$C_2$}
\psplot{0}{4.99}{\myhyperbola{5}}\rput(3.4,3.5){$C_1$}
\end{psclip}
\rput(3.5,1.1){Domain of stability}
\end{pspicture}
\caption{The hyperbolas $C_k$}\label{f:hyperbolas}
\end{figure}
Then the following holds in each of the following two cases.
\begin{enumerate}
\item $H_\alpha=L_2(\Omega,\C^2)$ and one of the following holds:
\begin{enumerate}
\item\label{ii:reactL2}
$\gamma\in[1/2,1)$ and either $d=1$, $\rho\le1$,
or $d=2$, $\rho<1$, or $d\ge3$, $\rho\le2/d$;
\item\label{ii:reactLp}
$D(A)$ is continuously embedded into $L_p(\Omega,\C)$,
$\rho<1-\frac2p$, and $\gamma\in(0,1)$ is sufficiently large;
\item\label{ii:reactHgamma}
$\rho\le1$, $\gamma\in(1/2,1)$, and $H_\gamma$ is continuously embedded into
$L_\infty(\Omega,\C)$;
\item\label{ii:reactW22}
$D(A)$ is continuously embedded into $W^{2,2}(\Omega,\C)$, and
either $d\le3$, $\rho\le1$, $\gamma\in(d/4,1)$,
or $d\ge4$, $\rho<4/d$, and $\gamma\in(0,1)$ is sufficiently large.
\end{enumerate}
\item $H_\alpha=W^{1,2}(\Omega,\C^2)$ and one of the following holds:
\begin{enumerate}
\item $d\le2$, $\rho>0$, $\gamma\in[0,1/2)$;
\item $d\ge3$, $\rho<4/(d-2)$, $\gamma\in[\max\Curly{0,\gamma_0},1/2)$,
where $\gamma_0$ is defined in~\eqref{e:gamma0} with $\sigma=\rho+1$.
\end{enumerate}
\end{enumerate}
For each $\varepsilon>0$ there is $\delta>0$ such that for each
$u_0\in H_\alpha$ with $\Norm{u_0}\le\delta$ there is a unique
$\gamma$-mild solution $u\in C([0,\infty),H_\alpha)$
of~\eqref{e:rctdiff},~\eqref{e:bdrys} with $u(0,\cdot)=u_0$,
$\Norm{u(t,\parameter)}_{H_\alpha}\le\varepsilon$ for all $t>0$
and $\Norm{u(t,\parameter)}_{H_\alpha}\to0$ exponentially as
$t\to\infty$.

We first note that~\iiref{ii:reactW22} is actually a special case
of~\iiref{ii:reactLp} and~\iref{ii:reactHgamma} by the Sobolev
embedding theorems and~\cite[Theorem~1.6.1]{Henry}, respectively.
Since $f_i$ is independent of $x$ and $t$, hypothesis~\eqref{e:g0est}
follows with $\sigma=\rho+1$ from~\eqref{e:fuvLip} and from the definition
of $f_i'$. Note also that the symmetry of $A$ implies
$D(A)=D(A^*)$ and $H_\gamma=\Hcc{\gamma}$.
The existence and uniqueness assertion follows from Proposition~\ref{p:Lipop}
or from Remark~\ref{r:CembedLip} in case~\iref{ii:reactHgamma}.
For the stability assertion, we apply Theorem~\ref{t:dgl}
or Remark~\ref{r:Cembedt} in case~\iref{ii:reactHgamma} with $r=\infty$
and $\sigma=1+r$. In view of Proposition~\ref{p:weakev},
it thus suffices to verify that there is $\lambda_0>0$ such that every
eigenvalue $\lambda$ of $A-B$ satisfies $\Real\lambda\ge\lambda_0$.
Under condition~\eqref{e:DS} the latter was verified
in~\cite{VaethQuittnerProc}. It can be shown by a similar calculation
that if $d_i>0$ violate~\eqref{e:DS} then there is an eigenvalue $\lambda$
of $A-B$ with $\Real\lambda\le0$ ($\lambda=0$ if
$(d_1,d_2)\in\bigcup_{k=1}^\infty C_k$). In this sense, the domain of stability
sketched in Figure~\ref{f:hyperbolas} is maximal.
\end{example}

Note that~\iiref{ii:reactW22} involves a strictly weaker requirement
concerning $\rho$ than~\iref{ii:reactL2} for every $d\ge2$. The embedding
required for~\iiref{ii:reactW22} holds in case $\Gamma_D=\emptyset$ or
$\Gamma_N=\emptyset$ if $\partial\Omega$ is sufficiently smooth.

The result obtained in~\cite{VaethQuittnerProc} concerning
Example~\ref{x:hyperbola} did not cover the case
$H_\alpha=L_2(\Omega,\C^2)$. Moreover, even in the case
$H_\alpha=W^{1,2}(\Omega,\C^2)$ and $d\ge3$,
the result in~\cite{VaethQuittnerProc} essentially needed
the more restrictive hypothesis $\rho\le2/(d-2)$ which is
(almost) by the factor~$2$ worse than our above requirement for that case.

\appendix
\section{On the Characterization of Kato Operators}

As an application of Theorem~\ref{t:Kato}, we obtain now a sufficient
criterion for Kato operators. In fact, in the following we give a necessary and
sufficient condition under which the particular scalar product~\eqref{e:bM}
is $A$-Kato.

Throughout this section, we consider the setting of Section~\ref{s:Hilbert}.
Recall that Proposition~\ref{p:Aassoc} implies in particular that
$A^{-1}\colon H\to H$ is bounded. It is well known (see
e.g.~\cite[Theorem~III.5.30]{Kato}) that this implies that also
$(A^*)^{-1}\colon H\to H$ exists and is bounded and is actually the
(bounded) Hilbert-space adjoint $(A^{-1})^*$, i.e.
\begin{equation}\label{e:Aminus}
(A^*)^{-1}=(A^{-1})^*\text.
\end{equation}

\begin{definition}
We call $A$ \emph{quasi-symmetric} if there are constants
$\alpha>-1$ and $\beta,M\ge0$ with
\begin{equation}\label{e:halfsymmAs}
\Real\Scalr{(A^*)^{-1}(Au+Mu),u}\ge\alpha\Abs u^2\quad\text{and}\quad
\Abs{(A^*)^{-1}Au}\le\beta\Abs u\quad\text{for all $u\in D(A)$.}
\end{equation}
If $M\ge0$ is given, we call $A$ \emph{$M$-quasi-symmetric} if there are
constants $\alpha>-1$, $\beta\ge0$ with~\eqref{e:halfsymmAs}.
\end{definition}

\begin{remark}
The larger $M$ is, the less restrictive condition~\eqref{e:halfsymmAs}
becomes. Indeed,~\eqref{e:Aminus} implies
\begin{equation}\label{e:goodpart}
\begin{aligned}
\Real\Scalr{(A^*)^{-1}u,u}&=\Real\Scalr{u,A^{-1}u}=
\Real\Scalr{A(A^{-1}u),A^{-1}u}\ge c\Norm{A^{-1}u}^2\ge0
\qquad\text{for all $u\in H$.}
\end{aligned}
\end{equation}
\end{remark}

\begin{remark}
If $A$ is symmetric, then~\eqref{e:goodpart} implies that $A$ is
$M$-quasi-symmetric with every $M\ge0$.
\end{remark}

Roughly speaking, estimates~\eqref{e:halfsymmAs} mean indeed that
$A$ is quantitatively almost symmetric in the sense that $(A^*)^{-1}A$
does not differ too much from the identity in a quantitative manner,
namely that it is ``almost'' accretive and bounded in $H$
(on the subspace $D(A)$).
The restriction $\alpha>-1$ may appear very strange at a first glance,
but it is the correct hypothesis for the following result:

\begin{proposition}\label{p:halfsymmetric}
For every $M\ge0$ the following assertions are equivalent.
\begin{enumerate}
\item $A$ is $M$-quasi-symmetric.
\item There are $\alpha>-1$ and $\widetilde\beta>0$ with
\begin{equation}\label{e:halfsymmA}
\Real\Scalr{Au+Mu,A^{-1}u}\ge\alpha\Abs u^2\quad\text{and}\quad
\Abs{\Scalr{Au,A^{-1}u}}\le\widetilde\beta\Abs u^2
\end{equation}
for all $u\in D(A)$.
\item The formula~\eqref{e:bM} defines an $A$-Kato scalar product on $V$.
\end{enumerate}
The relation of the largest possible constants $\alpha$ in~\eqref{e:halfsymmAs}
and~\eqref{e:halfsymmA} and $c_1$ in Proposition~\ref{p:a0} is given by
$2c_1=1+\alpha$.
\end{proposition}
\begin{proof}
For every $u,v\in D(A)$, we obtain from~\eqref{e:bM}, the definition of $A$,
and~\eqref{e:Aminus} that
\begin{equation}\label{e:bMcalc}
\begin{gathered}
2b_M(u,A^{-1}v)=a(u,A^{-1}v)+\overline{a(A^{-1}v,u)}+M\cdot\Scalr{u,A^{-1}v}\\
=\Scalr{Au+Mu,A^{-1}v}+\overline{\Scalr{v,u}}
=\Scalr{(A^*)^{-1}(Au+Mu),v}+\Scalr{u,v}\text.
\end{gathered}
\end{equation}
Hence, if~\eqref{e:bc1c2} or~\eqref{e:bc1c3} hold, then~\eqref{e:halfsymmAs}
or~\eqref{e:halfsymmA} hold with $\alpha\coloneqq2c_1-1$ and some
$\beta,\widetilde\beta>0$, respectively. Conversely, if~\eqref{e:halfsymmAs}
or~\eqref{e:halfsymmA} holds, then~\eqref{e:bMcalc} shows
that~\eqref{e:bc1c2} or~\eqref{e:bc1c3} hold with
$c_1\coloneqq(\alpha+1)/2$ and some $c_2,c_3>0$, respectively.
\end{proof}

\begin{theorem}\label{t:quasiKato}
If $A$ is quasi-symmetric, then $A$ is a Kato operator.
\end{theorem}
\begin{proof}
In view of Proposition~\ref{p:halfsymmetric}, the assertion follows
from Theorem~\ref{t:Kato}.
\end{proof}


\begin{thebibliography}{10}

\bibitem{AgranovichMixed}
Agranovich, M.~S., \emph{Mixed problems on a {Lipschitz} domain for strongly
  elliptic second-order systems}, Funkcional. Anal. i Prilo\v{z}en.
  \textbf{45} (2011), no.~2, 1--22, Engl. transl.: Funct. Anal. Appl.
  \textbf{45} (2011), no. 2, 81--98.

\bibitem{AgranovichGarding}
\bysame, \emph{Spectral problems in {Lipschitz} domains}, Sovrem. Mat. Fundam.
  Napravl. \textbf{39} (2011), 11--35, Engl. transl.: J. Math. Sci.
  \textbf{190} (2013), no. 1, 8--33.

\bibitem{Agranovich}
Agranovich, M.~S. and Selitskii, A.~M., \emph{Fractional powers of operators
  corresponding to coercive problems in Lipschitz domains}, Funkcional. Anal.
  i Prilo\v{z}en. \textbf{47} (2013), no.~2, 2--17, Engl. transl.: Funct.
  Anal. Appl. \textbf{47} (2013), no. 2, 83--95.

\bibitem{AmannExtrapol}
Amann, H., \emph{Nonhomogeneous linear and quasilinear elliptic and parabolic
  boundary value problems}, Function Spaces, Differential Operators and
  Nonlinear Analysis (Schmeisser, H.-J. and Triebel, H., eds.), Teubner,
  Springer, Leipzig, Wiesbaden, 1993, 9--126.

\bibitem{Amann}
\bysame, \emph{Linear and quasilinear parabolic problems}, vol.~I,
  Birkh\"{a}user, Basel, Boston, Berlin, 1995.

\bibitem{VaethAnalysis}
Appell, J. and V\"{a}th, M., \emph{Elemente der Funktionalanalysis},
  Vieweg \& {Sohn}, Braunschweig, Wiesbaden, 2005.

\bibitem{ArendtSurvey}
Arendt, W., \emph{Semigroups and evolution equations: Functional calculus,
  regularity and kernel estimates}, Handbook of Differential Equations.
  Evolutionary Equations (Dafermos, C.~M. and Feireisl, E., eds.), vol.~1,
  Elsevier, Amsterdam, 2004, 1--85.

\bibitem{ArendtPerturb}
Arendt, W. and Batty, C. J.~K., \emph{Forms, functional calculus, cosine
  functions and perturbations}, Perspectives in Operator Theory. Papers of the
  Workshop on Operator Theory, Warsaw, Poland, April 19--May 3, 2004
  (Warsaw) (Arendt, W., Batty, C. J.~K., Mbekhta, M., Tomilov, Y., and
  Zem\'{a}nek, J., eds.), Banach Center Publ., vol.~75, Polish Acad Sci., 2007,
  17--38.

\bibitem{EgertKatoMixed}
Egert, M., Haller-Dintelmann, R., and Tolksdorf, P., \emph{The {Kato}
  squareroot problem for mixed boundary conditions}, J. Funct. Anal.
  \textbf{267} (2014), no.~5, 1419--1461.

\bibitem{FiedlerVishik}
Fiedler, B. and Vishik, M.~I., \emph{Quantitative homogenization of analytic
  semigroups and reaction-diffusion equations with Diophantine spatial
  frequencies}, Adv. Differential Equations \textbf{6} (2001), no.~11,
  1377--1408.

\bibitem{HallerRehberg}
Haller-Dintelmann, R. and Rehberg, J., \emph{Maximal parabolic regularity for
  divergence operators including mixed boundary conditions}, J. Differential
  Equations \textbf{247} (2009), 1354--1396.

\bibitem{Henry}
Henry, D., \emph{Geometric theory of semilinear parabolic equations}, Lect.
  Notes Math., no. 840, Springer, Berlin, New York, 1981.

\bibitem{KatoPowersI}
Kato, T., \emph{Fractional powers of dissipative operators}, J. Math. Soc.
  Japan \textbf{13} (1961), 246--274.

\bibitem{KatoPowersII}
\bysame, \emph{Fractional powers of dissipative operators II}, J. Math. Soc.
  Japan \textbf{14} (1962), 242--248.

\bibitem{Kato}
\bysame, \emph{Perturbation theory for linear operators}, Springer, New York,
  1966.

\bibitem{VaethQuittner}
Kim, I.-S. and V\"{a}th, M., \emph{The Krasnosel'ski\u{\i}-Quittner formula
  and instability of a reaction-diffusion system with unilateral obstacles},
  Dynamics of Partial Differential Equations \textbf{11} (2014), no.~3,
  229--250.

\bibitem{KomatsuII}
Komatsu, H., \emph{Fractional powers of operators, II Interpolation
  spaces}, Pacific J. Math. \textbf{21} (1967), no.~1, 89--111.

\bibitem{KrasTop}
Krasnoselski\u{\i}, M.~A., \emph{Topological methods in the theory of nonlinear
  integral equations \textup{in Russian}}, Gostehizdat, Moscow, 1956, Engl.
  transl.: Pergamon Press, Oxford 1964.

\bibitem{Lunardi}
Lunardi, A., \emph{Analytic semigoups and optimal regularity in parabolic
  problems}, Birkh\"{a}user, Basel, Boston, Berlin, 1994.

\bibitem{MaziaSob}
Mazja, V.~G., \emph{Sobolev spaces}, Springer, Berlin, Heidelberg, 1985.

\bibitem{McLean}
McLean, W., \emph{Strongly elliptic systems and boundary integral equations},
  Univ. Press, Cambridge, 2000.

\bibitem{Ouhabaz}
Ouhabaz, E.~M., \emph{Analysis of heat equations on domains}, Princeton Univ.
  Press, Princeton, Oxford, 2005.

\bibitem{Pazy}
Pazy, A., \emph{Semigroups of linear operators and applications to partial
  differential equations}, Springer, New York, Berlin, Heidelberg, 1992.

\bibitem{Shamin}
Shamin, R.~V., \emph{Spaces of initial data for differential equations in
  Hilbert spaces and the Kato problem}, {Ulmer Seminare. Funktionalanalysis
  und Differentialgleichungen.} \textbf{7} (2002), 375--388.

\bibitem{Smoller}
Smoller, J., \emph{Shock waves and reaction diffusion equations}, Springer, New
  York, 1983.

\bibitem{Triebel}
Triebel, H., \emph{Interpolation theory, function spaces, differential
  operators}, North-Holland, Amsterdam, New York, Oxford, 1978.

\bibitem{VaethIdeal}
V\"{a}th, M., \emph{Ideal spaces}, Lect. Notes Math., no. 1664, Springer,
  Berlin, Heidelberg, 1997.

\bibitem{VaethSupAtomI}
\bysame, \emph{Continuity and differentiability of multivalued superposition
  operators with atoms and parameters I}, J. Anal. Appl. \textbf{31} (2012),
  93--124.

\bibitem{VaethQuittnerProc}
\bysame, \emph{Instability of Turing type for a reaction-diffusion system
  with unilateral obstacles modeled by variational inequalities}, Math. Bohem.
  (Prague, 2013), vol. 139, Proceedings of Equadiff 13, no.~2, 2014, 195--211.

\end{thebibliography}
\end{document}